\documentclass[a4paper,11pt]{article}

\usepackage[utf8]{inputenc}
\usepackage[british]{babel}
\usepackage{fancybox}
\usepackage[mono=false]{libertine}
\usepackage[T1]{fontenc}
\usepackage{amsthm}
\usepackage[normalem]{ulem}
\usepackage[cmintegrals,libertine]{newtxmath}
\usepackage[cal=euler, scr=boondoxo]{mathalfa}
\useosf
\usepackage{microtype}
\usepackage{numprint}
\usepackage[margin=2.5cm]{geometry}
\usepackage{multirow}

\usepackage{tikz}

\usepackage{graphicx}
	\graphicspath{{./images/}}

\usepackage{pdfpages}

\usepackage[font=sf]{caption}
\usepackage{hyperref}
\hypersetup{
	colorlinks=true,
		citecolor=blue!60!black,
		linkcolor=red!60!black,
		urlcolor=green!40!black,
		filecolor=yellow!50!black,
	breaklinks=true,
	pdfpagemode=UseNone,
	bookmarksopen=false,
}

\usepackage{enumitem}

\numberwithin{equation}{section}

\newtheorem{proposition}{Proposition}[section]

\newtheorem{theorem}[proposition]{Theorem}

\newtheorem{open}[proposition]{Question}
\newtheorem{lemma}[proposition]{Lemma}

\theoremstyle{definition}

\newtheorem*{definition}{Definition}
\newtheorem{remark}[proposition]{Remark}

\theoremstyle{remark}

\renewcommand\P{\mathbb{P}}
\newcommand\E{\mathbb{E}}

\newcommand{\ce}{\boldsymbol{\mathcal{E}}}

\newcommand\R{\mathbb{R}}

\newcommand{\bea}{\begin{eqnarray}}
\newcommand{\eea}{\end{eqnarray}}

\def\void{}
\def\labelmark{}

\newenvironment{formula}[1]{\def\labelname{#1}
\ifx\void\labelname\def\junk{\begin{displaymath}}
\else\def\junk{\begin{equation}\label{\labelname}}\fi\junk}%
{\ifx\void\labelname\def\junk{\end{displaymath}}
\else\def\junk{\end{equation}}\fi\junk\labelmark\def\labelname{}}

{\ifx\void\labelname\def\junk{\end{array}\end{displaymath}}
\else\def\junk{\end{array}\right.\end{equation}}
\fi\junk\labelmark\def\labelname{}\def\junk{}
}

\newcommand{\beq}{\begin{formula}}
\newcommand{\eeq}{\end{formula}}
\newcommand{\beqv}{\begin{formula}{}}

\newenvironment{romenumerate}[1][0pt]{
\addtolength{\leftmargini}{#1}\begin{enumerate}
 }{\end{enumerate}}

\newcounter{oldenumi}
{\setcounter{oldenumi}{\value{enumi}}
\begin{romenumerate} \setcounter{enumi}{\value{oldenumi}}}
{\end{romenumerate}}

\begingroup
  \divide\count255 by 60
  \multiply\count255 by -60
  \advance\count255 by \time
  \ifnum \count255 < 10 \xdef\klockan{\the\count1.0\the\count255}
  \else\xdef\klockan{\the\count1.\the\count255}\fi
\endgroup

\newcommand\REM[1]{{\raggedright\texttt{[#1]}\par\marginal{XXX}}}

\newcommand\eps{\varepsilon}

\def\rompar(#1){\textup(#1\textup)}    

\def\xexp(#1){e^{#1}}

\newcounter{CC}
\newcounter{cc}

\newcommand\doi{\mathbb{D}([0,1])}

\newcommand\marginal[1]{\marginpar{\raggedright\parindent=0pt\tiny #1}}

\newcommand\shred{\boldsymbol{\mathcal{S}}}

\newcommand\multset{\mathcal{M}}

\newcommand\randmult{\mathscr{M}^\mu_n}
\newcommand\randdual{\mathscr{G}^\mu_n}

\newcommand\up{\mathrm{u}}
\newcommand\down{\mathrm{d}}

\newcommand{\tth}{\mathtt{h}}
\newcommand{\ttb}{\mathtt{b}}

\newcommand{\one}{\hbox{\rm 1\kern-.27em I}}

\title{Stable shredded spheres and  causal random maps \\ with large faces}

\author{Jakob Bj\"ornberg\thanks{Chalmers and the University of Gothenburg, Sweden. E-mail: jakob.bjornberg@gu.se},~ Nicolas Curien\thanks{Laboratoire de Mathématiques d’Orsay
		Université Paris-Saclay and Institut Universitaire de France. \newline E-mail: nicolas.curien@gmail.com}~ and Sigurdur \"Orn Stef\'ansson\thanks{Science Institute, University of Iceland. E-mail: sigurdur@hi.is }} 

\begin{document}

\maketitle

\begin{abstract} 
	We introduce a new familiy of random compact metric spaces 
	$\shred_\alpha$ for $\alpha\in(1,2)$, which we
	call \emph{stable shredded spheres}. They are constructed from
        excursions of $\alpha$-stable Lévy processes on $[0,1]$
        possessing no negative jumps. Informally, viewing the graph of
        the Lévy excursion in the plane, each jump of the process is
        "cut open" and replaced by a circle and then all points on the
        graph at equal height which are not separated by a jump are
        identified. We show that the shredded spheres arise as scaling
        limits of models of causal 
	random planar maps with large faces 
introduced by Di Francesco and Guitter.
We also establish that their Hausdorff
        dimension is almost surely equal to $\alpha$. Point
        identification in the shredded spheres is intimately connected
        to the presence of decrease points in stable 
spectrally positive L\'evy processes as studied by Bertoin in the 90's.
\end{abstract}

\begin{figure} [h]
	\centerline{\scalebox{0.60}{\includegraphics{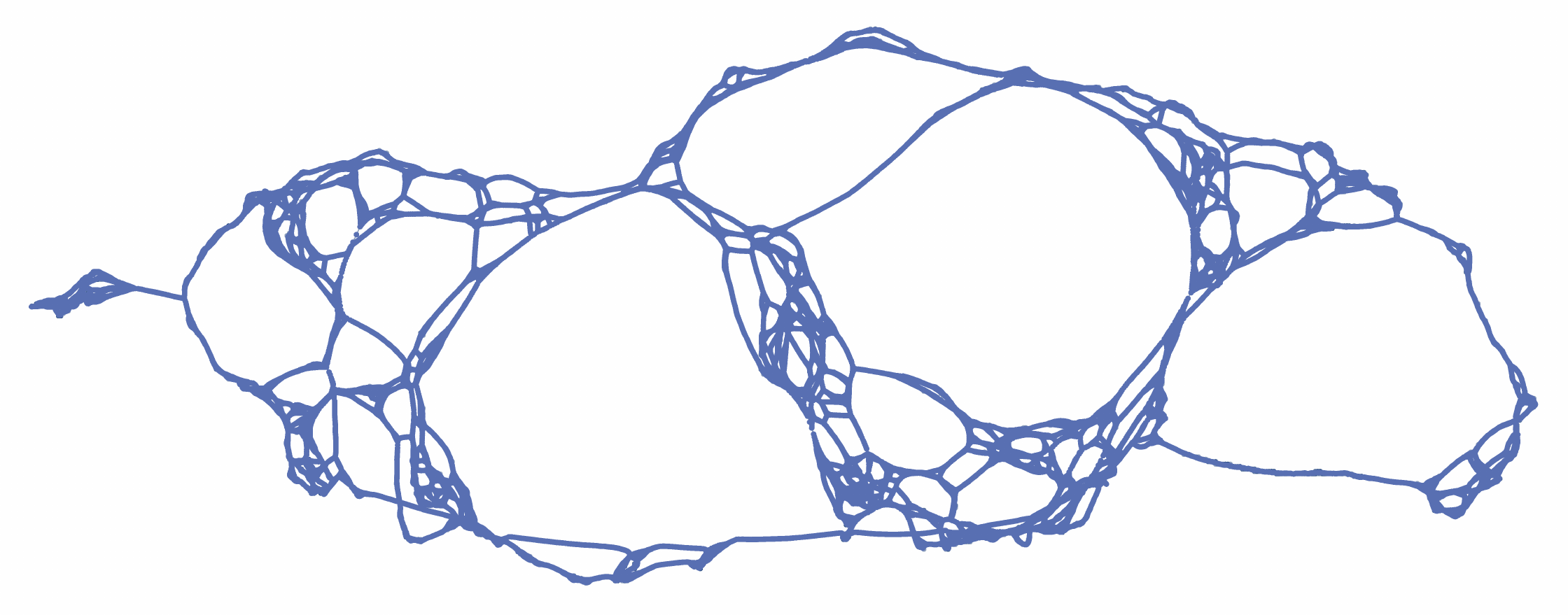}}} 
	\caption{A simulation of a large random causal map $\randdual$ which approximates a stable shredded sphere 
		$\shred_\alpha$ with $\alpha=1.8$. Note that the drawing is not isometric and the time slices are (approximately) vertical.} \label{f:simulation}
\end{figure} 

\section{Introduction}\label{s:intro}
\subsection{Random planar geometries}
In recent years, there has been considerable progress in the study of
random two dimensional surfaces, in particular random planar maps and
their scaling limits.  Planar maps are finite connected graphs
properly drawn on the two-sphere viewed up to continuous
deformations. A natural way to choose a random planar map is to
pick uniformly from the set of triangulations of the sphere with
a fixed number of triangles. Due to bijections of Bouttier, Di
Francesco and Guitter \cite{bouttier:2004} (building on
\cite{cori:1981,schaeffer:1998}) the random triangulations may be
viewed as random labelled trees which are further in bijection with
random processes. 
These correspondences have been used
to prove many deep results about random planar maps. Most notable are
the independent proofs of Miermont \cite{miermont:2013} (uniform
quadrangulations) and Le Gall \cite{legall:2013} (uniform
triangulations and $2p$-angulations) that properly rescaled random
maps converge, in the Gromov--Hausdorff sense, towards a compact
metric space called \emph{the Brownian map} (this type of convergence
is often referred to as a scaling limit). Subsequently, it has been
shown that the Brownian map 
is the scaling limit for a large class of different discrete models, 
see e.g.~\cite{ab:2019,bettinelli:2014,marzouk:2018}.  
It has also been constructed within the Gaussian Free Field by 
Miller and Sheffield \cite{miller-sheffield-QLE}.

Another family of random compact metric spaces was introduced in
\cite{legall:2011} and referred to as \emph{stable maps}. We mention
these here since they are relevant to the model which we introduce in
the current paper. The stable maps arise as scaling limits, at least
along subsequences, of random planar maps defined by assigning weights
to the faces such that the degree of a typical face is in the domain
of attraction of a stable law with index $\alpha\in(1,2)$. See
\cite[Section 5.2]{curien-stFlour} for more details.  One of
the motivation for studying stable maps is that they appear in
statistical mechanical models on planar maps
\cite{budd:2018,legall:2011,richier:2018}.

\subsection{Causal maps} Models of random planar maps became popular
in high energy physics in the 1980's and 1990's due to their
connections to random matrix models and Liouville quantum gravity (see
e.g.~\cite{ambjorn:1997}). In the latter, one quantizes gravity via a
path integral which is formally defined as an integral over a suitable
set of space-time geometries with a weight which is given in terms of
the action of the classical gravity theory. In order to make sense of
this integral, one discretizes the geometries and the problem
essentially boils down to performing the sum over a suitable set of
planar maps having a fixed size and an appropriate weight which is
derived from the gravity action. The weights may be viewed as a
probability measure on this set of planar maps by normalizing. When
restricted to triangulations of the sphere, we recover the model of
uniform triangulation of the sphere described above whose scaling
limit is the Brownian map. This model is referred to by physicists as
\emph{dynamical triangulations} (DT), see Figure~\ref{f:cdt}. 

An important question in the model of DT is what constitues a suitable
set of space-time geometries in the path integral. Ambjørn and Loll
\cite{ambjorn:1998} argued that causality should be taken into account
and imposed a `causal' condition on the planar maps appearing in the
discretized path integral. This was initially formulated for
triangulations roughly as follows. The sphere is divided into lines of
constant latitudes (including the poles) and each adjacent pair of
latitudes is triangulated with a single layer of triangles, see
Figure~\ref{f:cdt}.  This model is referred to as \emph{causal dynamical
  triangulations} (CDT). Although a large uniformly random CDT is a
non-trivial object,
it turns out that its scaling limit is simply a
line segment
\cite[Theorem 1]{curien:2017}. 
The reason is that the latitutes collapse to points in
the limit since a pair of typical points on a given latitude are a
factor of a logarithm closer to each other than to (say) the south
pole. The Gromov-Hausdorff limit is thus a useless
concept to study properties of large CDT.   

\begin{figure} [h]
	\centerline{\scalebox{0.47}{\includegraphics{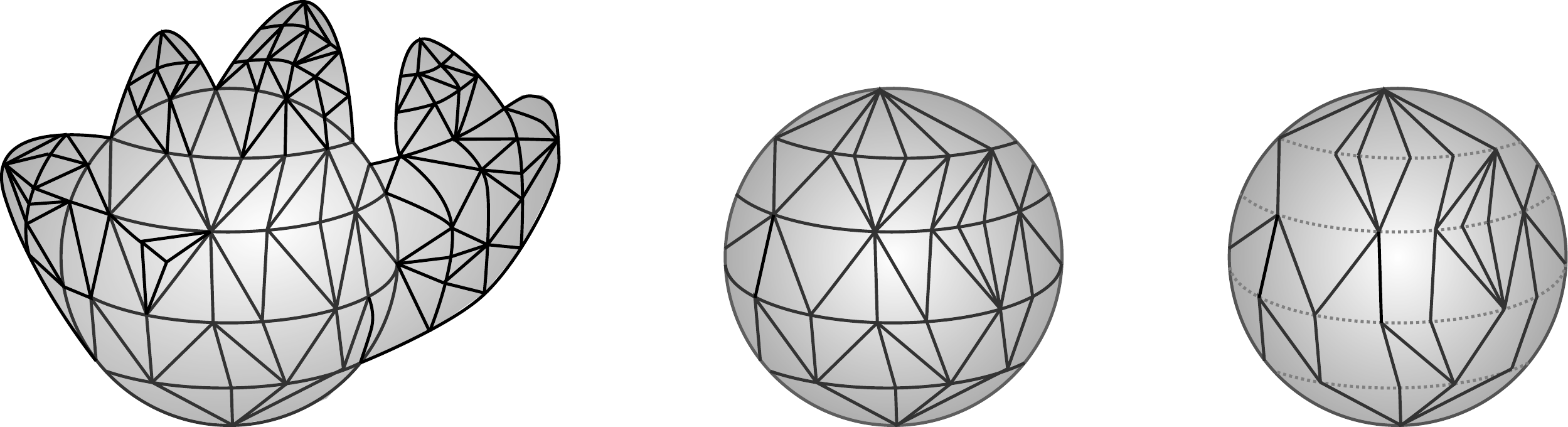}}} 
	\caption{Left: DT. Middle: CDT. Right: Generalized CDT.} \label{f:cdt}
\end{figure} 

In this paper, building on the work of Di Franscesco and Guitter \cite{difrancesco:2002}, we study a model of generalized CDT 
where the faces are allowed to have arbitrary degrees and 
they are assigned weights such that the degree of a typical face is in
the domain of attraction of an $\alpha$-stable law, $\alpha\in(1,2)$
(see Figure~\ref{f:cdt} for an illustration and the next section for a
precise statement).  This is analogous to the definition of the stable
maps which was mentioned above. As we will see these random maps will
have a non-trivial scaling limit.

\subsection{The hard multimer model and its scaling limits}

Let us define our model of hard multimers, which is closely related to
 one originally defined by Di Francesco and Guitter 
 \cite{difrancesco:2002}. We will use it to define the model of causal
 maps through a bijection. 

Consider the semi-infinite cylinder represented by the set $C =
[0,1]\times \mathbb{R}_+$ with the vertical boundaries identified. We
draw equally spaced lines $[0,1]\times {i}$, 
for $i \in \{0,1,2, ...\}$ in
$C$ which we will sometimes refer to as \emph{time-slices}. 
A \emph{hard multimer configuration} 
is defined by drawing vertical lines of
integer length (zero 
 or larger) in $C$ so that their endpoints lie on
time-slices and so that no two such lines intersect, see
Figure~\ref{f:multimers}. A vertical line is referred to as a
\emph{multimer}. We will assume that a multimer configuration is
\emph{connected} in the sense that the projection of all the multimers
onto ${0}\times \mathbb{R}_+$ is an interval of the form $[0,k]$,
$k\geq 0$. One of the endpoints which lie on the first time-slice is
singled out and called the root. 
We consider  a multimer configuration as a planar map whose
vertex set is the set of intersection points of horizontal
 and
vertical lines and the edges are the line segments connecting these
points. Let $\multset$ denote the set of connected multimer
configurations and $\multset_n$ the subset of those with exactly $n$
vertices.

To remove any arbitrariness in the drawing of a multimer configuration
we will use the following convention referred to as a \emph{left
staircase boundary condition} in \cite{difrancesco:2002}: The root is
drawn furthest to the left and the leftmost multimer which intersects
layer $[0,1]\times [i,i+1]$, $i>0$, is to the right of the leftmost
multimer which intersects the layer immediately below. See
Figure~\ref{f:multimers} for an illustration of these concepts.

\begin{figure} [h]
 \centerline{\scalebox{0.80}{\includegraphics{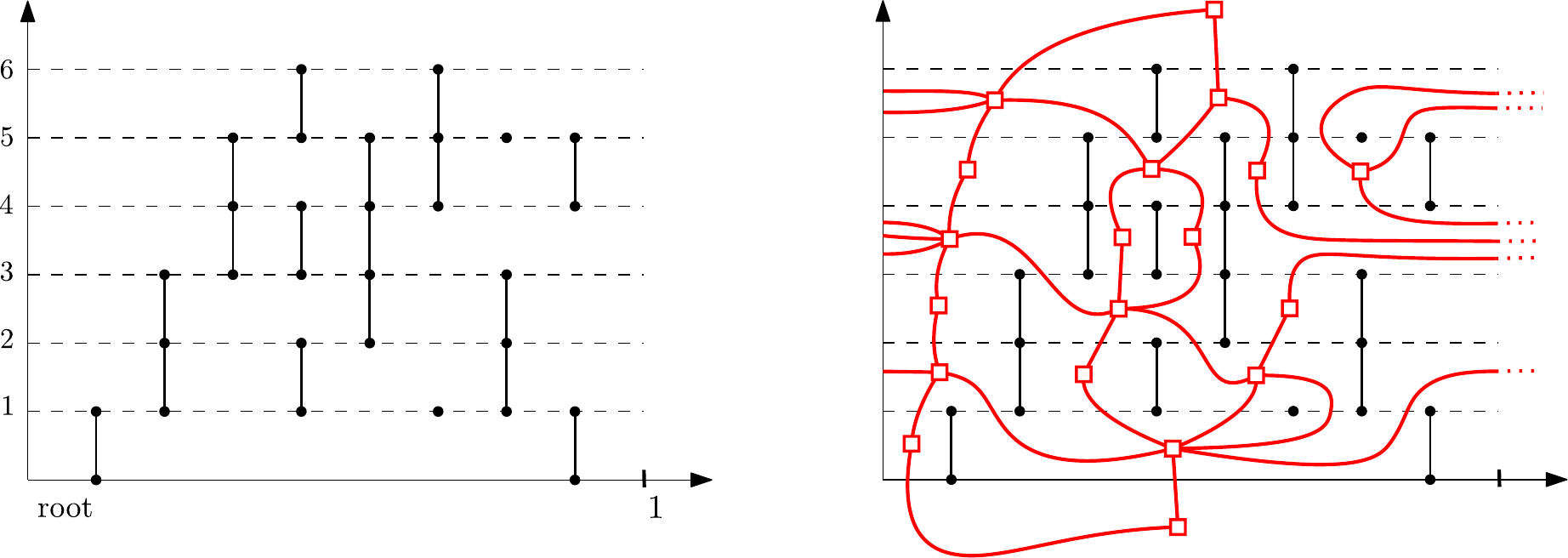}}} 
 \caption{An example of a configuration $M$ of hard multimers on the
   cylinder (element of $\multset_{30}$). The configuration is
   connected and rooted. It is drawn with a left staircase boundary
   condition. On the right, the vertical dual $G$ where dual edges
   crossing multimers are forbidden. } \label{f:multimers} 
 \end{figure} 

Given a multimer configuration $M\in \multset$ we define its
\emph{vertical dual} $G$ by placing a vertex in each face of $M$, as
well as a vertex below the lowest time-slice and a vertex above the
highest time-slice, and an edge between vertices such that the
corresponding faces share a horizontal boundary. 
The graph distance on $G$ is thus the minimal number of
horizontal time-slices to cross to go from one point to the other by
avoiding all the multimers (such a path can go around the cylinder if
needed). The vertical dual will be called a \emph{causal map}. Its
faces are in one to one correspondence with the multimers and the
degree of a face equals 2 times the number of vertices in the
corresponding multimer. Note that a causal triangulation is recovered
by letting all multimers have length 1 and by including the horizontal
edges in the dual.

Let $\mu$ be a probability measure on  $\{-1,0,1,2,\dotsc\}$. 
We define a probability distribution $\mathbb{P}^\mu_n$ on the set of
multimer configurations $\multset_n$ by assigning to each $M \in
\multset_n$ the probability 
\begin {equation} \label{eq:defBoltzmann}
 \mathbb{P}_n^\mu(M) = \frac{1}{Z_n}\prod_{m \text{ multimer in } M} \mu_{|m|} \left( \mu_{-1}\right)^{|m|}
\end {equation} 
where $|m|$ denotes the length (number of edges) 
of a multimer and $Z_n$ is the appropriate normalization. 
The
  presence of the factor $(\mu_{-1})^{m}$ is also a convention of
  normalization and is included here to make the link with random walk
  clearer.  This model is related to that of Di Francesco and
  Guitter \cite{difrancesco:2002}, see Section \ref{sec:DG}.  Denote a
random multimer configuration distributed by $\mathbb{P}_n^\mu$ by
$\randmult$. The associated random vertical dual graph will be denoted
by $\randdual$ and we call it a \emph{random causal map.} 
We regard $\randdual$ as a metric space with metric given by the
graph distance.

When  $\mu$ has mean zero and has finite variance, the result of \cite{curien:2017} should hold and we believe that $\randdual$ falls in the universality class of so
called \emph{generic causal dynamical triangulations}, see Section
\ref{sec:comments}.
In this work we suppose that $\mu$ is centered but has infinite
variance. Specifically we suppose that
\begin {equation}\label{stable}
\mu([k,\infty)) \sim |\Gamma(1-\alpha)|^{-1} k^{-\alpha}\qquad \text{with} \quad \alpha\in (1,2).
\end {equation}
So that $\mu$ is in the strict domain of attraction of the spectrally positive $\alpha$-stable random variable. Then an  interesting random compact metric space appears as the scaling
limit of $\randdual$  for each $\alpha$:

\begin {theorem}[Shredded spheres as scaling limits]\label{thm:conv}  There is a family of random compact
metric  spaces $(\shred_\alpha : 1 < \alpha < 2)$, called the
\emph{stable shredded spheres} such that if  $\mu$ is centered and
satisfies \eqref{stable} then we have  
\[
n^{-1/\alpha} \cdot \randdual \xrightarrow[n\to\infty]{(d)}
\shred_\alpha,
\]
in distribution for the Gromov--Hausdorff topology.
\end {theorem}
Here $c\cdot G$ means that graph distances in $G$ are
multiplied by the constant $c$. For background on the
Gromov--Hausdorff topology, we refer to \cite{burago}. A
simulation of a large $\randdual$ is shown in 
Figure~\ref{f:simulation}.  

\subsection{Shredded spheres as a metric gluing of two trees} 

A precise definition of $\shred_\alpha$ is given in 
Section \ref{sec:defshred}, but roughly speaking we construct
$\shred_\alpha$ as follows.  Start with an excursion 
$\ce=(\ce(t))_{t\in[0,1]}$ of
an $\alpha$-stable L\'evy process with no negative jumps.    As it is well-known, one may
construct a real tree by ``gluing the underside'' of $\ce$
(identifying points at the same height which are not separated by a
local minimum), see \cite{duquesne:2008}.  In the same way, another real tree may be obtained by
applying the same construction to the reflected excursion
$(-\ce(-t)+\max\ce)_{t\in[0,1]}$ 
(with cyclically adjusted
time-paramterization).  We obtain the shredded sphere $\shred_\alpha$
by applying both identifications \emph{simultaneously};  that is, we
``glue''  both the underside and overside of $\ce$. In a more precise sense, we endow $[0,1]$ with a random pseudo-metric $ \mathbf{D}^\ast$ obtained as 
\[ 
 \mathbf{D}^{\ast}(s,t) = \inf \left\{ \sum_{i=1}^{n}  \mathbf{D}^{\up}(a_{i},b_{i})+  \mathbf{D}^{\down}(b_{i},a_{i+1}): 
{s=a_{1},b_{1},  \dotsc, a_{n}, b_{n}, a_{n+1}=t}\right\},
\] 
where $ \mathbf{D^\up}$ and $ \mathbf{D}^\down$ are the psuedo-metrics
on $[0,1]$ coding for the tree `below' and `above' $ \ce$.  Such type
of constructions is classic in the theory of random maps, e.g. the
Brownian map \cite{legall:2013} or in the mating of trees theory
\cite{mating-of-trees}.  But it is usually
challenging to show that $ \mathbf{D}^\ast$ appears as the limiting
metric of rescaled discrete models: in the case of the Brownian map,
this corresponds to the breakthroughs of Le Gall \cite{legall:2013} and
Miermont \cite{miermont:2013}. Also,
contrary to the case of random planar maps, our models do not possess
re-rooting invariance which was crucial in the aforementioned
results. Our salvation will come from noting that the large multimers
create impassable barriers in the scaling limit.  More
precisely, we can define another pseudo metric 
$ \mathbf{V}(s,t)$ on $[0,1]$ obtained by the minimal vertical 
total variation to
go from $s$ to $t$ by avoiding the jumps of $ \ce$. Our main 
result is that $ \mathbf{V}= \mathbf{D}^\ast$ almost surely for 
$\ce$, see Theorem \ref{thm:V=D*} (this is not a deterministic
statement, see Remark \ref{ex:counterex} for a counterexample). 
With these ingredients at hands, the proof of Theorem \ref{thm:conv} is
rather easy. We also show (Theorem \ref{l:vd0}) that 
\[
\mathbf{V}(s,t)=0 \quad \iff \quad  \mathbf{D}^\up(s,t)=0   
\mbox{ or }   \mathbf{D}^\down(s,t)=0,
\]
thus characterizing the point
identifications in $ \shred_\alpha$. Our methodology, based on the use
of large faces to control the metric, will be applied to the study of
stable maps \cite{legall:2011} in a forthcoming work \cite{curien:2021}.  
A key step is a powerful bootstrapping argument introduced in a
different context by Gwynne and Miller \cite{gwynne-miller}.

We also study various properties of the $\alpha$-stable shredded
sphere which are connected to fine properties of the $\alpha$-stable
L\'evy process. For example, the local behavior of such processes
entails that $ \shred_\alpha$ is almost surely of Hausdorff dimension
$\alpha$ (Proposition \ref{prop:dim}).
Our proof of this closely mimics the case of the
$\alpha$-stable looptrees \cite{curien:2014}. 
Furthermore, our
shredded spheres $  \shred_\alpha$ possess large faces which are the
image of the jumps of $ \ce$ in our construction. We show using the
work of Bertoin \cite{bertoin:1994} on increase points of L\'evy
processes that there are large faces which touch each other. 
Unfortunately, as in the
case of the mating-of-trees theory of Duplantier, Miller and Sheffield
\cite[Question 11.2]{mating-of-trees} we were unable to 
decide whether the graph of faces in $ \shred_\alpha$ is connected (two
  faces are adjacent if they touch each other).  
More generally we leave the following question open:
\begin{open} Is it possible to characterize the topology of $ \shred_{\alpha}$?  
\end{open}
\medskip

\textbf{Acknowledgments:} 
The first author acknowledges support from
\emph{Vetenskapsr{\aa}det}, Grants 2015-0519 and 2019-04185,
and is grateful for the
hospitality at Université Paris-Sud Orsay.
The second author acknowledges supports from
ERC ``GeoBrown'' as well as the grant \texttt{ANR-14-CE25-0014} ``ANR
GRAAL''. The third author acknowledges support from the Icelandic
Research Fund, Grant Number: 185233-051, and is grateful for the
hospitality at Université Paris-Sud Orsay and at Chalmers. 
We also thank the two anonymous referees for crucial remarks that
helped clarifying the paper.

\tableofcontents

\section{Discrete encodings and bounds on distances}

In this section we recall the coding of (random) multimer configurations using
(random) walks. This will motivate our definition and construction of
the shredded sphere $\shred_\alpha$ using the normalized excursion
$\ce$ of an $\alpha$-stable spectrally positive L\'evy process. We
refer the reader to \cite{bertoin:1996,Cha97} for the construction and
basic properties of (excursions of) stable L\'evy processes.  We shall
also prove lower and upper bounds for asymptotic distances in
$\randdual$ reducing the proof of  Theorem \ref{thm:conv} to 
properties of $ \ce$. 

\subsection{Coding by a walk}\label{sec:walk}
Let $\randmult$ be our random multimer configuration with $n$ vertices
and  $\randdual$ be its vertical dual. We associate with $\randmult$ a
random walk $ \mathcal{E}_n$ as follows. For each multimer $m \in
\randmult$ 
except for the root, 
draw a horizontal segment from its bottom point going to
the left and stop at the first vertex encountered (blue lines in
Figure~\ref{f:walk}). 
The bottom point of the root is 
instead connected   vertically
to $\varnothing^\up$, an additional
point below the lowest time slice. This construction results in a
tree. 
As in  \cite{difrancesco:2002}
we now perform the left-to-right contour walk around this tree,
starting at the bottom of the root,
going up the multimers in single steps (of sizes 
$k \in \{0,1,2,\dotsc \}$) and going down by steps of $-1$ when needed
(horizontal blue edges are traversed without taking a step in the
corresponding walk). 
The resulting countour function, 
 with time indexed by $\{0,1,\dotsc,n+1\}$,
 is denoted by $\ce_n$;  it starts at level $0$ and 
stops at level $-1$ at
time $n+1$.  We extend the domain of  $\ce_n$ to the interval $[0,n+1]$
by linear interpolation between the integer points.
The construction, which is a bijection between multimer
configurations and excursions of walks whose jumps are in $\{-1,0,1,2,3,...\}$, 
should be clear on  Figure~\ref{f:walk}.  
Note that ${\mathcal E}_n$ records the vertical distances to
the lowest point of $ \randmult$.

\begin{figure}[h!]
 \centerline{\scalebox{0.80}{\includegraphics{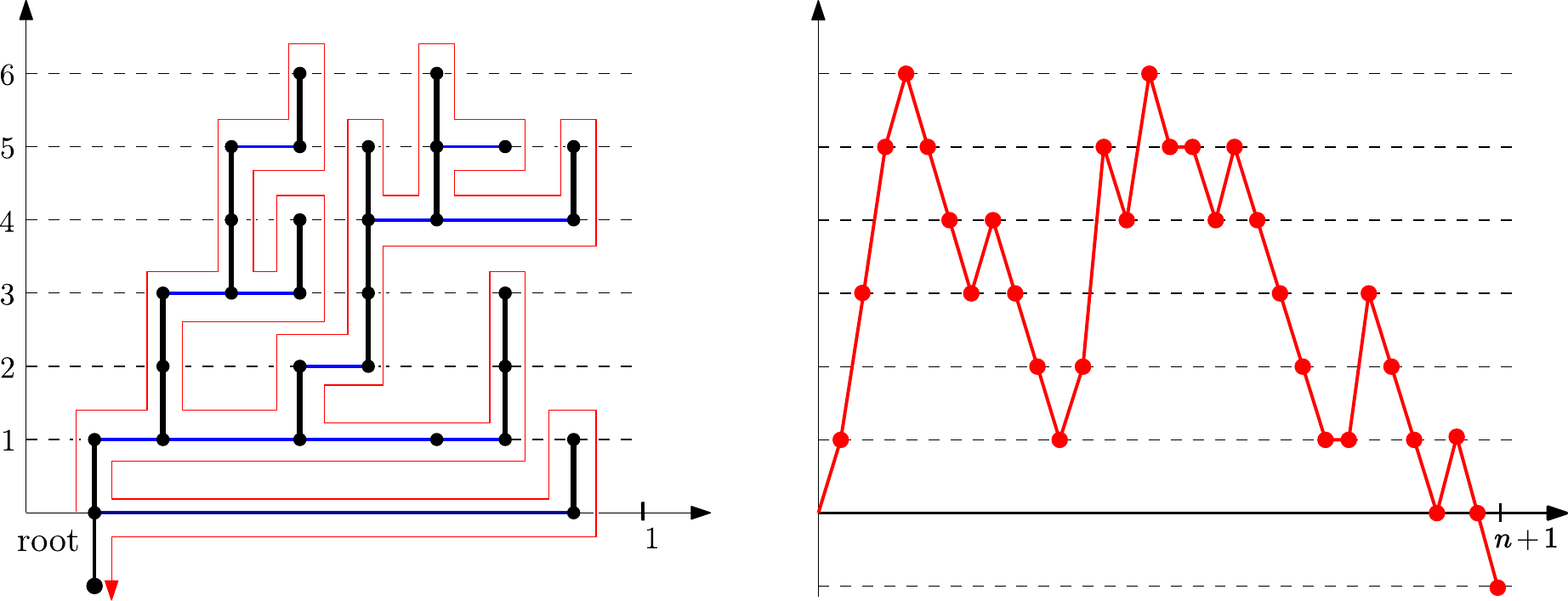}}} 
 \caption{Illustration of the coding of a multimer configuration by an
   excursion of a random walk whose negative steps are of size $-1$
   only.} \label{f:walk} 
 \end{figure}

The underlying multimer configuration can be 
identified with the closed set formed by 
 \begin{eqnarray} \label{def:slits}
\mathrm{Slits}_n=
\bigcup_{0\leq k\leq n-1}
\{k+1\}\times 
\big[\mathcal{E}_n(k),
\mathcal{E}_n(k+1)\big] \subset [0,n] \times \mathbb{R}_+.
  \end{eqnarray}
(Only multimers contribute to the union since $[j,j-1]=\varnothing$.)
Notice that we consider here the cylinder of width $n+1$ whereas we
considered before a cylinder of width $1$, but this horizontal
dilation is irrelevant in what follows. In this representation, the
vertices of $\randdual$ except for the top and 
bottom vertices can be identified with $(t, \mathcal{E}_{n}(t))$
for those times $t \in \mathcal{R}_{n}$ where  
\begin{equation} \label{eq:Rncor}
\mathcal R_n=\{k+\tfrac12: 
{\mathcal E}_n(k+1)={\mathcal E}_n(k)-1\} \cap [0,n+1].
\end{equation}

\begin{figure}[h!]
 \centerline{\scalebox{0.87}{\includegraphics{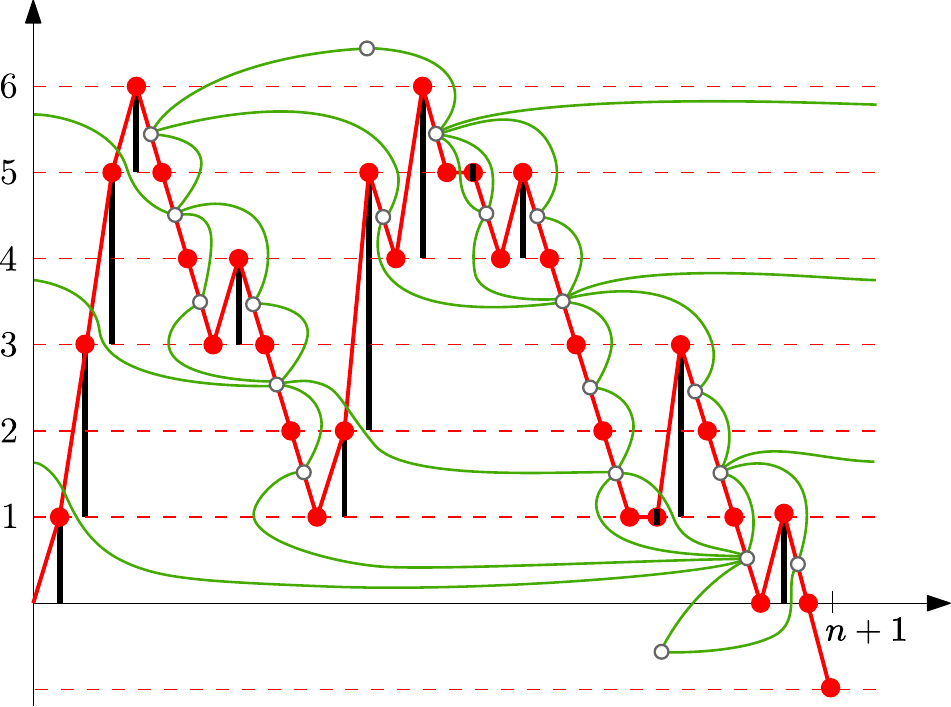}}} 
 \caption{The graph $\mathcal{G}_n^\mu$ (green) drawn on the cylinder
   $[0,n+1]\times[-1, \infty)$, with vertical sides identified, together
   with the excursion $\mathcal{E}_n$ (red) and $\mathrm{Slits}_n$ (black).}
\label{fig:G-walk} 
 \end{figure} 
 It is also more convenient to start our cylinder at height $-1$ to
 include the vertex $\varnothing^{ \mathrm{u}}$, and we shall always
 imagine that the vertices $ \varnothing^{ \mathrm{u}}$ and $
 \varnothing^{ \mathrm{d}}$ are respectively placed at height $-1/2$
 and $ \mathrm{sup} \mathcal{E}_{n}+1/2$. See Figure \ref{fig:G-walk}
 for an illustration.

It will be convenient to associate an element of $\mathcal R_n$
to each $t\in[0,n]$,
 hence we define
\begin{equation}\label{pwc}
t^\circ=\left\{\begin{array}{ll}
\max\{r\in\mathcal R_n: r\leq t\}, &
\mbox{if } t\geq\min \mathcal R_n, \\
\max \mathcal R_n, & \mbox{otherwise}.
\end{array}\right.
\end{equation}
Note that if $t \in \mathcal{R}_{n}$ then $ \mathcal{E}_{n}(t) +
\frac{1}{2}$ is the graph distance between $\varnothing^{ \mathrm{u}}$
and the corresponding vertex in $ \randdual$ via \eqref{eq:Rncor}.
We also consider the piecewise
constant excursion
$\overline{\mathcal{E}}_n(t):= \mathcal{E}_n(t^\circ)$.
Note that $\overline{\mathcal{E}}_n(\cdot)$ is càdlàg and agrees with
$\mathcal{E}_n(\cdot)$ on the set $\mathcal R_n$, and that 
$\overline{\mathcal{E}}_n(t)=\tfrac12$ for $0\leq t<\min \mathcal R_n$.

\subsubsection{Law of the walk under Boltzmann measures} 
\label{sec:DG} 

Let us compute the law of $ \mathcal{E}_{n}$ in the
case when the multimer configuration has a ``Boltzmann'' law with a
generic weight sequence. We do so in order to shed light on the
normalization conditions used in the introduction. If $
\mathbf{q}=(q_{k} : k \geq 1)$ is a sequence of non-negative weights
(not necessarily of sum $1$) we can define a probability distribution
$ \mathbb{P}^{ \mathbf{q}}_{n}$ on the set of multimer configurations
of size $n$ by putting 

\begin {equation} \label{eq:defBoltzmannbis}
 \mathbb{P}_n^{ \mathbf{q}}(M) \propto \prod_{m \text{ multimer in } M} q_{|m|}, \quad \mbox{ for each }M \in \mathcal{M}_{n},
\end {equation} 
where $|m|$ denotes the length (number of edges) of a multimer. Notice that since we restrict to configurations with $n$ vertices in total, for any $\lambda, \xi >0$ we have 
$$\prod_{m \text{ multimer in } M} q_{|m|} = (\lambda \xi)^{-n} \prod_{m \text{ multimer in } M} \left(\xi q_{|m|} \lambda^{|m|+1}\right) \cdot \xi^{|m|}.$$ We shall assume that $\lambda,\xi>0$ are chosen so that $ (\mu_{k} : k \geq -1)$ defined by 
$$\mu_{k} =\xi q_{k} \lambda^{k+1},  \mbox{ for } k \geq 0, \quad
\mbox{and} \quad \mu_{-1} = \xi,$$ is a probability measure. By doing
so, we are back to the setting \eqref{eq:defBoltzmann} used in the
introduction. The additional freedom on $(\lambda,\xi)$ allows us, in
generic situations, to further fix the mean of $\mu$ to be $1$ (see \cite[Section 4]{janson:2012} for further details in the equivalent context of simply generated trees).  Now
it is easy to see that the push forward of this distribution on the
set of excursion paths by the above bijection gives the law of a
$\mu$-random walk conditioned to first hit $-1$ at time $n+1$. Our
standing assumptions \eqref{stable} is then to ask that the step
distribution $\mu$ of that random walk is critical and in the domain
of attraction of the $\alpha$-stable law. 

Our model is related to that of 
Di Francesco and Guitter 
 \cite{difrancesco:2002} as follows:
setting $\mu_0=0$ and $t_i=\mu_i(\mu_{-1})^i$ for $i\geq1$,
we have
\begin{equation}
\sum_{n\geq 0} Z_n=\lim_{T\to\infty} Z_T(\{t_i\}),
\end{equation}
where the right-hand-side uses the notation of
\cite[Section 2.2]{difrancesco:2002} for the partition-function of
multimer-configurations on $T$ time-slices.  Thus, we restrict the
number of vertices $n$ rather than the number of time-slices $T$.

 \subsection{A first scaling limit} The above coding of multimer
 configurations by walks enables us to define right away the scaling
 limit for the coding excursion which will be our entry door to the
 continuous world. \medskip

 By \emph{$ \alpha$-stable L\'evy process} we mean a  stable spectrally
positive L\'evy process $X$ of index $\alpha$, normalized so that for
every $\lambda>0$, 
$$\E[\exp(-\lambda X_t)]=\exp(t \lambda^\alpha).$$
 Equivalently, the
L\'evy measure $\Pi$ of $X$ is supported by $ \mathbb{R}_+$ and we
have 
\begin{eqnarray} 
\Pi([r, \infty))&=& | \Gamma(1-\alpha)|^{-1}r^{-\alpha}, \quad \mbox{ for } r >0.   \label{eq:levymeasure} \end{eqnarray}
The trajectories of $X$ a.s.\ belong to the Skorokhod space $D(\R_+, \R)$ of
right-continuous with left limits (c\`adl\`ag) functions, endowed with
the Skorokhod topology (see \cite[Chap. 3]{billingsley}).  The dependence of
$X$ in $\alpha$ will be implicit in the whole paper.  We consider (see
\cite{Cha97}) the normalized excursion $\ce$ 
of $X$ above its infimum. 
Notice that 
$\ce$ is a.s.\,a random c\`adl\`ag function on $[0,1]$ such
that $ \ce(0)=  \ce(1)=0$ and $ \ce(s)>0$ for
every $s \in(0,1)$.

 \begin{proposition}[Scaling limit of the coding
   excursion]   \label{prop:firstscaling}
If $\mu$ is centered and satisfies \eqref{stable} then the random
c\`adl\`ag excursion $ \overline{\mathcal{E}}_n$ associated to
$\randmult$ by the above construction satisfies  
 \begin{eqnarray}
  \label{eq:convcontour}
\big(n^{ -1/\alpha} \cdot \overline{\mathcal E}_n(nt)\big)_{t\in[0,1]}
\to \big(\ce (t)\big) _{t\in[0,1]}
  \end{eqnarray}
in the sense of weak convergence with respect  to the Skorokhod metric
on $\mathbb D[0,1]$.
 \end{proposition}
 \begin{proof} 
By the discussion at the end of the last section, 
the law of $ \mathcal{E}_n$ is that of a random walk with
i.i.d.~increments of distribution $\mu$ conditioned on reaching $-1$
for the first time at $n+1$.  The result then follows from  an
adaptation of the arguments in 
\cite[Lemmas~4.5-6]{duquesne:2003} 
(to deal with the interpolation between integer-points
using the random function $t^\circ$ in \eqref{pwc}). 
\end{proof} 

Proposition \ref{prop:firstscaling} already gives the order of the
diameter of  $\randdual$, 
since $ \overline{\mathcal E}_n+1/2$ records the distances to
$\varnothing^{ \mathrm{u}}$ in $\randdual$.  Using Skorokhod's embedding
theorem, we can suppose that our multimer configurations $\randmult$
are coupled in such a way that the convergence of
Proposition \ref{prop:firstscaling} 
is almost sure.   This then means
that there are (random) continuous increasing bijective
functions $\psi_n: [0,1] \to [0,n]$ 
so that  
 \begin{eqnarray} \label{eq:cvinfty} 
(n^{-1/\alpha}\cdot \overline{ \mathcal{E}}_n(
   \psi_n (t)) : t \in [0,1]) 
\xrightarrow[n\to\infty]{} 
(\ce(t) :  t\in [0,1])  
\end{eqnarray} 
almost surely for the supremum norm on
 $[0,1]$ and so that 
$\|\tfrac1n \psi_n - \mathrm{Id}\|_\infty \to  0$. 
In the following we shall write $\widetilde{ \mathcal{E}}_n  = n^{-1/\alpha}
\overline{ \mathcal{E}}_n \circ \psi_n$.

Our goal is to use \eqref{eq:cvinfty} to show almost
sure convergence of the rescaled graphs $\randdual$,
thereby establishing our main result Theorem \ref{thm:conv}.
 The road will be quite long and we shall start with discrete
 upper and lower bounds on the distances in $ \randdual$ which pass to the limit.
 
 \subsection{Two trees and an upper bound}

Recall the encoding of $\randmult$ by $ \mathcal{E}_n$
and let us denote
 by $\varnothing^\down$ and $\varnothing^\up$ the top and bottom
 vertices of $\randdual$ respectively.  We define two trees, which we
 call the \emph{up-tree} (rooted at $\varnothing^\up$) and
 \emph{down-tree} 
(rooted at $\varnothing^\down$), as follows, see
Figure~\ref{fig:twotrees}.
The up-tree is obtained by connecting each vertex of $\randdual$
(except $\varnothing^\down$)
to its leftmost neighbor in the time-slice below, and the down-tree is
obtained by connecting each vertex of $\randdual$
(except $\varnothing^\up$) to its rightmost neighbor in the time-slice
above.  Both trees are subgraphs of $\randdual$ but their union does
not yield $\randdual$ in general. 

\begin{figure}[h!]
 \begin{center}
  \centerline{\scalebox{0.80}{\includegraphics{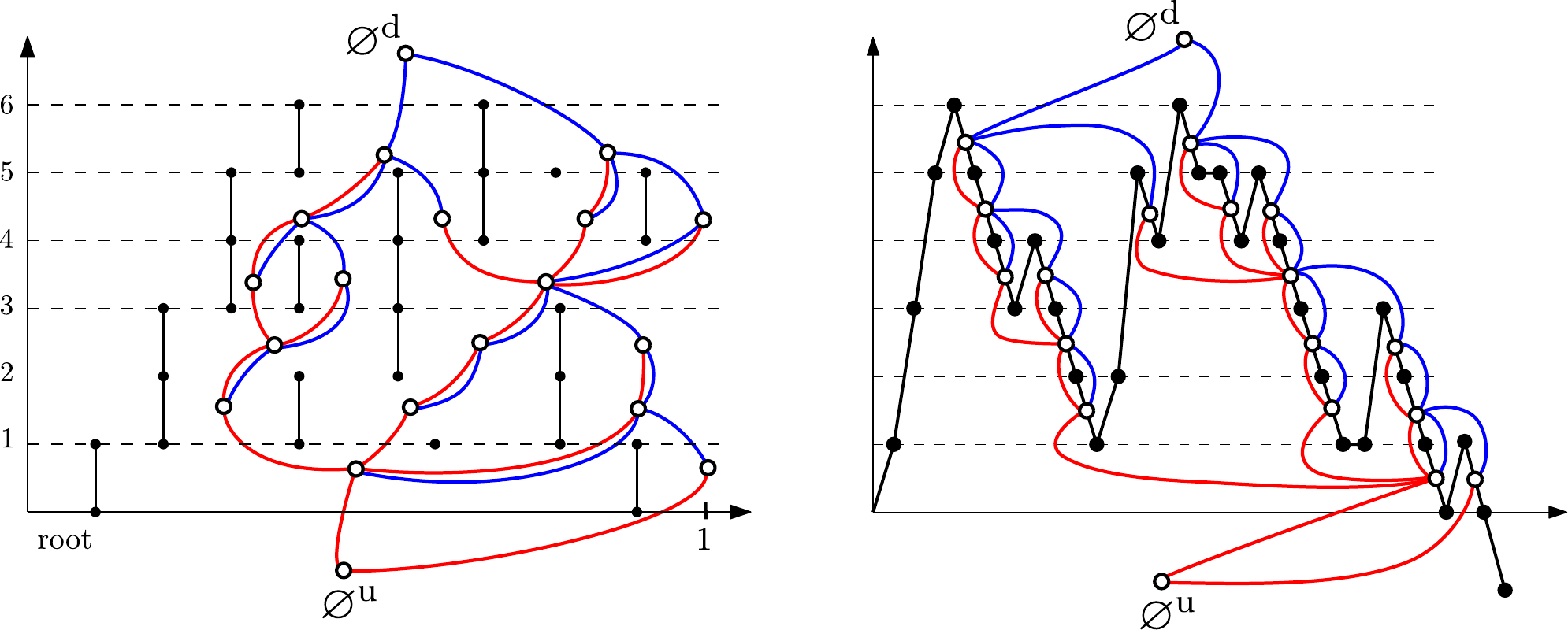}}}
 \caption{Illustration of the two trees lying inside $ \randdual$, the
   up-tree is in red and is rooted at the bottom vertex
   $\varnothing^\up$, the down-tree is in blue and is rooted at the
   top vertex $\varnothing^\down$. Both trees are subgraphs of $
   \randdual$ and the distances in those trees are easily expressed
   using $ \mathcal{E}_n$. On the right it is shown how to recover the
   trees from the walk which corresponds to the multimer
   configuration.   \label{fig:twotrees}} 
 \end{center}
 \end{figure}

Recall from the last section that the vertices of $\randdual$,
  except for $\varnothing^{ \mathrm{u}}$ and $\varnothing^{
    \mathrm{d}}$, are in correspondence with $ \mathcal{R}_{n}$.
Distances within these trees can be expressed using the excursion 
$\mathcal{E}_n$:
 if $u,v \in \mathcal{R}_n$ the distances $D^\up(u,v)$ (resp.~$
 D^\down(u,v)$) in the up-tree (resp.~down-tree) between the vertices
 associated with $u$ and $v$ in $\randdual$ are given by
\begin{equation}
\begin{split}
D^\up_n(u,v)&=\mathcal{E}_n(u)+\mathcal{E}_n(v)-
2\min_{w\in[u\wedge v,u\vee v]} \mathcal{E}_n(w)+1-\delta_{u,v},\\
D^\down_n(u,v)&=1-\delta_{u,v}+
2\min\big\{\max_{w\in[u\wedge v,u\vee v]} \mathcal{E}_n(w),
\max_{w\in[0,u\wedge v]\cup [u\vee v,n]} \mathcal{E}_n(w)\big\}
-\mathcal{E}_n(u)-\mathcal{E}_n(v).
\end{split}
\end{equation}
(The definition of $D^\down$ can be interpreted in terms of the
``mirror'' excursion of $\mathcal{E}_n$, we explain this in the continuous
setting below, see \eqref{eq:mirrored}.)

Let us write $D_n(u,v)$ for the distance in $ \randdual$
between (the vertices associated with) $u,v \in
\mathcal{R}_n$. Clearly we have both $D_n(u,v) \leq D_n^\up(u,v)$ and
$ D_n(u,v) \leq D_n^\down(u,v)$. In order to see $D_n, D_n^\up$ and
$D_n^\down$ as pseudo-distances on $[0,1]$, similar to what we did for
$ \mathcal{E}_n$ just after Proposition \ref{prop:firstscaling}, we
shall introduce for $s,t \in[0,1]$ 
\[ 
\widetilde{D}_n(s,t) = D_n(
\psi_n(s)^\circ, \psi_n(t)^\circ), \quad
\widetilde{D}^{\up/\down}_n(s,t) = D^{\up/\down}_n( \psi_n(s)^\circ,
\psi_n(t)^\circ),
\]
where $\psi_n$ is the time-change in \eqref{eq:cvinfty}. 
With these notations, the distance metric in  $ \randdual$
between vertices (excepting $\varnothing^\up$ and $\varnothing^\down$)
is given by the quotient of $[0,1]$ by $
(\widetilde{D}_n =0)$ endowed with the projection of $
\widetilde{D}_n$. 

We now move on to the continuous setting and define
accordingly  for $s,t \in [0,1]$ 
\[  
\mathbf{D}^{\up}(s,t) = \boldsymbol{\ce}(s)+ \ce(t) - 
2 \inf_{u \in [s \wedge t, s \vee t]} \ce(u).
\]
To help the reader understand the next definition notice that we
always have $\inf_{u \in [0,s \wedge t] \cup [s \vee t,1]}
\ce(u)=0$ and so we could have replaced $\inf_{u \in [s \wedge
t, s \vee t]} \ce(u)$ by $\max\{\inf_{u \in [s \wedge t, s
\vee t]} \ce(u), \inf_{u \in [0,s \wedge t] \cup [s \vee t,1]}
\ce(u)\}$ in the last definition. The second pseudo distance
is the same as the first but for the ``mirror'' excursion 
$-\ce(-\cdot)+\max \ce$ and is given for $s,t \in [0,1]$ by

\begin{equation}\label{eq:mirrored}  
\mathbf{D}^{\down}(s,t) = 2\min \left\{ \sup_{u \in [s \wedge t, s
    \vee t]} \ce(u), \sup_{u \in [0,s \wedge t] \cup [s \vee t,1]}
  \ce(u)\right\} - \ce(s) - \ce(t).
\end{equation}
It is easy to check using \cite[eq. (3)]{duquesne:2008} 
that $ \mathbf{D}^{\up}$ and
$ \mathbf{D}^\down$ are pseudo distances. 
Furthermore the quotient metric spaces
$([0,1]/\{ \mathbf{D}^i=0\}, \mathbf{D}^{i})$ for 
$i \in \{\up, \down\}$ 
are both real trees, which 
we refer to as the trees coded by $ \ce$ and by 
$-\ce(-\cdot)+\max \ce$, respectively. In a different context, such
trees were consdidered by Lambert \cite{lambert:2010} as special cases
of random real trees which belong to the family of so-called
\emph{splitting trees}. 
In our case we   ``glue these trees
together'', that is we consider $ \mathbf{D}^{\ast}$ the largest pseudo distance on
$[0,1]$ that is both smaller than $ \mathbf{D}^{\up}$ and 
$\mathbf{D}^{\down}$ viz 
\begin{equation}\label{eq:d-star}
 \mathbf{D}^{\ast}(s,t) = \inf \left\{ \sum_{i=1}^{k}  \mathbf{D}^{\up}(a_{i},b_{i})+  \mathbf{D}^{\down}(b_{i},a_{i+1}): 
{s=a_{1}, b_{1},  \dotsc, a_{k}, b_{k}, a_{k+1}=t}\right\}.
\end{equation}
Our asymptotic upper bound on the distances in $ \randdual$ (similar
to the bound $D \leq D^*$ in the case of the Brownian map, see
\cite[Proposition 3.3]{legall:2007}) reads now as follows:
\begin{proposition}[Upper bound on distances] \label{prop:upbound}
Almost surely, under the coupling in \eqref{eq:cvinfty}
\[ 
\limsup_{n \to \infty} \sup_{s,t \in [0,1]} \left(n^{-1/\alpha}
  \widetilde{D}_n(s,t) - \mathbf{D}^\ast(s,t)\right) \leq 0.
\]
\end{proposition} 
\begin{proof} We start by noting from \eqref{eq:cvinfty} and by
    definition of $ \widetilde{D}_{n}^{\mathrm{u}/ \mathrm{d}}$ and $
    \mathbf{D}^{\mathrm{u}/ \mathrm{d}}$ that we have  
\[ 
n^{-1/\alpha} \cdot \widetilde{D}_{n}^{\mathrm{u}/ \mathrm{d}}
\xrightarrow[n\to\infty]{a.s.} \mathbf{D}^{\mathrm{u}/ \mathrm{d}},
\]
uniformly over $[0,1]^{2}$. 
Using \eqref{eq:cvinfty} and properties of c\`adl\`ag
  functions, it is easy to see that for any given $\eps>0$ there is an
  $N$ such that for all large enough $n$, for any $s \in [0,1]$, we
  can find $s' \in \mathcal{D}(N)=\{k2^{-N}:k=0,1,\dotsc,2^N\}$ such
  that 
\[
 n^{-1/\alpha} \widetilde{D}_n^{ \mathrm{u}}(s,s') \leq \varepsilon/4 
\quad \mbox{ and }\quad 
\mathbf{D}_n^{ \mathrm{u}}(s,s') \leq \varepsilon/4.
\] 
Using the bounds $ \widetilde{D}_n \leq \widetilde{D}_n^{ \mathrm{u}}$
as well as  $\mathbf{D}^* \leq \mathbf{D}^{ \mathrm{u}}$ and the
triangle inequality, it follows that for all large enough $n$ we have 
 \begin{eqnarray} \label{eq:finipoints}
\sup_{s,t \in [0,1]} \left(n^{-1/\alpha}
  \widetilde{D}_n(s,t) - \mathbf{D}^\ast(s,t)\right) \leq
\eps+
\max_{s,t \in \mathcal{D}(N)} \left(n^{-1/\alpha}
  \widetilde{D}_n(s,t) - \mathbf{D}^\ast(s,t)\right),
  \end{eqnarray}
 Now, for each fixed $s_0,t_0$ we claim that $ \limsup_{n \to \infty} n^{-1/\alpha} \widetilde{D}_n(s_0,t_0) \leq \mathbf{D}^*(s_0,t_0)$. To see this, for a given $\eta>0$, fix a finite sequence
$a_i,b_i$ such that
\[
\mathbf{D}^{\ast}(s_0,t_0) \geq 
\sum_{i=1}^{k}  \mathbf{D}^{\up}(a_{i},b_{i})+  \mathbf{D}^{\down}(b_{i},a_{i+1}) - \eta.
\]
Using that 
\[
\widetilde{D}_n(s_0,t_0) \leq 
\sum_{i=1}^{k}  \widetilde{D}_n^{\up}(a_{i},b_{i})+  
\widetilde{D}^{\down}_n(b_{i},a_{i+1})
\]
and the fact that
$n^{-1/\alpha} \widetilde{D}_n^{\up/\down}(s,t) \to  
\mathbf{D}^{\up/\down}(s,t)$ we deduce the claim by letting $\eta \to
0$. Since there is a finite number of pairs of points $(s,t) \in
\mathcal{D}(N)^2$, we can take the $\limsup$ and ${n \to \infty}$ in
\eqref{eq:finipoints} and get the statement of the proposition by
letting $ \varepsilon \to 0$. 
\end{proof} 

\subsection{Lower bound on distances} \label{sec:lowerbound}

We now turn our attention to a lower bound on distances in 
$\randdual$ for which the role of the large multimers becomes
crucial. 
The rough idea is that large multimers form obstacles when
travelling horizontally, allowing us to bound  distances
from below by a vertical height variation. 

Recall the setup of the last section and  let us imagine that
  $\randdual$ is drawn on the cylinder $[0,n+1]\times[-1, \infty)$ on
  top of the coding walk $ \mathcal{E}_{n}$ as in
  Figure~\ref{fig:G-walk}. More precisely, the vertices of $\randdual$
  are the points $(t, \mathcal{E}(t))$ for $t \in \mathcal{R}_{n}$,
  together with the additional vertices $ \varnothing^{ \mathrm{u}}$
  and $ \varnothing^{ \mathrm{d}}$ which are respectively at height
  $-1/2$ and $ H+1/2$ where $H= \max \mathcal{E}_{n}$. Recall the
  definition of $\mathrm{Slits}_{n}$ from \eqref{def:slits}. Finally,
  the edges of the graph $\randdual$ can also be drawn on the
  cylinder, such that the edges do not intersect $ \mathrm{Slits}_{n}$
  and such that the vertical coordinate is \emph{monotone} along each
  edge, see Figure \ref{fig:G-walk}.

Consider two times in $ \mathcal{R}_n$ 
corresponding  to two vertices $a$ and $b$
of $\randdual$, and consider a geodesic path in $ \randdual$ going from $a$
to $b$. In the above embedding,  such a path may be seen as a continuous curve $\gamma$
from $[0,1]$ to  the cylinder
$[0,n+1] \times [-1,\infty)$ avoiding $ \mathrm{Slits}_{n}$. Since two adjacent vertices in $ \randdual$ necessarily are at height difference $1$ and since edges have monotone vertical variation, we can write
\[ 
D_n(i+\tfrac12, j + \tfrac12) = \mathrm{VarHt}(\gamma),
\]
where $\mathrm{VarHt}(\gamma)$ is the total variation of the vertical
coordinate of the path 
$\gamma = ((\gamma_x(t), \gamma_y(t)) : 0 \leq t \leq 1)$, that is 
\[ 
\mathrm{VarHt}(\gamma) = \sup_{0=t_0<t_1<t_2<\cdots <
  t_n=1}\sum_{i=0}^{n-1}\left|\gamma_y(t_{i+1})-\gamma_y(t_i)\right|.
\]

Let us mimic the above construction in the continuous world using the
excursion $ \mathcal{E}$.  For $s,t \in [0,1]$ let
 \begin{eqnarray} \label{eq:defV}
\mathbf{V}(s,t) = \inf_{\gamma : (s, \ce(s)) \to (t, \ce(t))}
\mathrm{VarHt}(\gamma),
 \end{eqnarray}
where the infimum is taken over all continuous paths $\gamma$ on the
cylinder $[0,1] \times \mathbb{R}_+$ with the vertical boundaries
identified, going from $(s, \ce(s))$ to $(t, \ce(t))$ and which 
does not cross any segments of 
\[ 
\mathbf{Slits}( \ce) = \bigcup_{i\geq 0} \{t_i\} \times [ \ce(t_i-) ,
\ce(t_i)]
\] 
where $(t_i : i \geq 0)$ is an enumeration of the jumps of $
\ce$. Above we say that  $\gamma$
  \emph{crosses}  a vertical segment $V=\{\tau\} \times [a,b]$ if there exist 
  times $g,d\in[0,1]$ 
so that $\gamma(t) \in \{\tau\} \times (a,b)$
for all $t\in[g\wedge d,g\vee d]$
 and so that any neighbourhood of $g$
contains a time where $\gamma$ is to the left of $V$, 
and any neighbourhood of $d$
contains a time where $\gamma$ is to the right of $V$.
Notice   that $\mathbf{V}(s,t)$ is almost surely finite and
that $\mathbf{V}(s,t) \geq | \ce(s)-\ce(t)|$.

We state a useful lemma which uses only basic analysis: 
\begin{lemma} \label{lem:min} 
Almost surely,
for any $s ,t \in [0,1]$ there exists a
  continuous path $\gamma : (s, \ce(s)) \to (t, \ce(t))$ which does not cross any segment of 
  $\mathbf{Slits}( \ce)$ and  such that $ \mathrm{VarHt}(\gamma) =
  \mathbf{V}(s,t)$ and where the $x$-component of $\gamma$ 
is monotone (on the cylinder).
\end{lemma}
\begin{proof} 
The lemma is deterministic and works if we replace $\ce$ by any càdlàg
function. Let us first see why we can restrict ourselves to paths
$\gamma = (\gamma_x,\gamma_y)$ whose $x$-coordinate $\gamma_x$ is
monotone (recall that we are on the 	cylinder). 
	
	Assume that $\gamma_x(0) \leq \gamma_x(1)$ and let $t' = \sup_{t\in[0,1]} \{ \gamma_x(t) = \gamma_x(0)\}$. Let us assume for concreteness that $\gamma_x(t) \geq \gamma_x(0)$ for all $t \in [t',1]$. This means that the path $\gamma$ does not wrap around the cylinder after time $t'$ and we will show that we may modify it such that its $x$-coordinate is monotonically increasing. Define a new path $\gamma^+ = (\gamma_x^+,\gamma_y)$ where
	\begin{align*}
		\gamma_x^+(t) = \min\bigg\{\sup_{s\in[0,t]} \gamma_x(s),\gamma_x(1)\bigg\}.
	\end{align*}
	The path $\gamma^+$ is by construction continuous, its
        $x$-coordinate is monotonically increasing and it has the same
        endpoints and the same height variation as $\gamma$. It
        remains to argue that it does not cross any slits. Assume the
        opposite, then there is a slit $\{\tau\}\times[a,b]$ and times
        $g \leq d$ such that $\gamma^+(t) \in \{\tau\}\times (a,b)$
        for all $t\in [g,d]$ and there is a $\delta >0$ such that
        $\gamma^+_x(t) < \tau$ for $t\in (g-\delta,g)$ and
        $\gamma^+_x(t) > \tau$ for $t\in (d,d+\delta)$. Let $g' =
        \inf\{t\leq d~:~ \gamma_x([t,d]) = \{\tau\}\}$. Since $d$ is a
        point of increase of $\gamma^+_x$ it holds that $\gamma_x(d) =
        \gamma_x^+(d) = \tau$ and thus the infimum is over a non-empty
        set. Since $\gamma_x(t) \leq \gamma_x^+(t) < \tau$ for $t\in
        (g-\delta,g)$ it holds that $g'\geq g$. Therefore, $\gamma(t)
        \in \{\tau\}\times(a,b)$ for all $t\in [g',d]$. Now, for any
        $\eps >0$ there is a $t \in (g'-\eps,g')$ such that
        $\gamma_x(t) < \tau$ (otherwise, $g'$ could be made smaller)
        and a $t'\in (d,d+\eps)$ such that $\gamma_x(t) > \tau$ (since
        $d$ is a point of increase of $\gamma_x^+$). We have thus
        shown that $\gamma$ crosses a slit which is a contradiction. 
	
	We now restrict ourselves to paths for which the
        $x$-coordinate is monotone. Consider a sequence of such paths
        $\gamma_n = (\gamma_{x,n},\gamma_{y,n})$ such that
        $\operatorname{VarHt}(\gamma_n) \to \mathbf{V}(s,t)$ as $n\to
        \infty$. Since the height variation converges, $\gamma_{n,y}$
        are of a uniformly bounded variation. Since $\gamma_{x,n}$ are
        monotone and uniformly bounded, they are also of a uniformly
        bounded variation. Therefore we may assume, using
          continuity and by reparameterizing, that the sequence is
        $C$-Lipschitz for some $C>0$. By Helly's selection theorem
        (see e.g.~ \cite{kolmogorov:1970}, Theorems 4 and 5 in Section
        36) $\gamma_n$ has a subsequential limit which we denote by
        $\gamma$, and  
	\begin{align*}
		\operatorname{VarHt}(\gamma) \leq 
\lim_{n\to\infty}\operatorname{VarHt}(\gamma_n) = \mathbf{V}(s,t).
	\end{align*}
	Moreover, the Lipschitz condition guarantees that $\gamma$ is continuous. 
	Finally, it is easy to see that $\gamma$ does not cross any
        slits since none of the $\gamma_n$ did.
\end{proof}

\begin{proposition}[Lower bound on distances]  \label{prop:lowbound}
Almost surely, under the coupling in \eqref{eq:cvinfty}
\[ \liminf_{n \to \infty} \inf_{s,t \in [0,1]} 
\left(n^{-1/\alpha} \widetilde{D}_n(s,t) - \mathbf{V}(s,t)\right) 
\geq 0.
\]
\end{proposition} 
\begin{proof}
As in Proposition \ref{prop:upbound} we use that, given $\eps>0$, there
is $N$ such that for all large enough $n$ we have
\[
\inf_{s,t \in [0,1]} 
\left(n^{-1/\alpha} \widetilde{D}_n(s,t) - \mathbf{V}(s,t)\right) 
\geq -\eps+\min_{s,t \in \mathcal{D}(N)} 
\left(n^{-1/\alpha} \widetilde{D}_n(s,t) - \mathbf{V}(s,t)\right).
\]
Moreover, for each such $s$, $t$ and $n$ there is a continuous
path $\gamma_{n,\eps}$ from 
$(s,\widetilde{\mathcal{E}}_n(s))$ to $(t,\widetilde{\mathcal{E}}_n(t))$ which does not cross 
$\mathbf{Slits}(\widetilde{\mathcal{E}}_n)$ such that 
$n^{-1/\alpha} \widetilde{D}_n(s,t)\geq 
-\eps+\mathrm{VarHt}(\gamma_{n,\eps})$. 
We also define for $\eps>0$ and $0\leq \delta \leq \eps$, the set 
\begin{align*}
\mathbf{Slits}_\eps^\delta( \ce) = \bigcup_{i\geq 0} \left\{\{t_i\} \times [ \ce(t_i-)+\delta/2 ,
\ce(t_i)-\delta/2] : \ce(t_i)-\ce(t_i^-) > \eps\right\}
\end{align*}
which consists of the slits of length larger than $\eps$ which are
furthermore truncated from the top and bottom by $\delta/2$. Denote
the $x$-coordinates of the slits in
$\mathbf{Slits}_\eps^0(\widetilde{\mathcal{E}}_n)$, by $t_{i,n}$,
$i=1,\ldots,m_{\eps,n}$ and of the slits in
$\mathbf{Slits}_\eps^0(\ce)$ by $t_{i}$, $i=1,\ldots,m_\eps$.  
Since $\widetilde{\mathcal{E}}_n$ converges uniformly to $\ce$ we may
assume, by choosing $n$ large enough, that $m_{\eps,n} = m_\eps$ and
(by permuting the jump-times) that $t_{i,n} = t_{i}$ and 
\[ 
|\widetilde{\mathcal{E}}_n(t_i)  - \ce(t_i)| < \eps/2, \qquad |  \widetilde{\mathcal{E}}_n(t_i^-) - \ce(t_i^-)| < \eps/2
\]
for all $1\leq i \leq  m_\eps$.
Arguing as in Lemma \ref{lem:min}, we may find a subsequence along which the $\gamma_{n,\eps}$  converge towards a continuous curve $\gamma_\eps$ which does not cross 
$\mathbf{Slits}_\eps^\eps(\ce)$ and letting $\eps \to 0$ we may find a sequence among the $\gamma_\eps$ which converges towards a continuous curve $\gamma$ which does not cross $\mathbf{Slits}(\mathcal{E})$ and such that 
\[
\liminf_{\eps \to 0} \liminf_{n\to\infty} \mathrm{VarHt}(\gamma_{n,\eps}) \geq \liminf_{\eps \to 0}\mathrm{VarHt}(\gamma_{\eps}) \geq \mathrm{VarHt}(\gamma) \geq V(s,t)
\]
which concludes the proof.
\end{proof}

\section{The $\alpha$-stable shredded spheres and 
$ \mathbf{V}=  \mathbf{D}^\ast$} 
Recall the definition of the two (random) pseudo-distances $
  \mathbf{D}^{\ast}$ given in \eqref{eq:d-star} and $ \mathbf{V}$
  given in \eqref{eq:defV}. 
In this section we establish our main result:
\begin{theorem} \label{thm:V=D*} 
For any $\alpha \in (1,2)$, almost surely  for $ \ce$ we 
have $\mathbf{V} = \mathbf{D}^\ast$.
\end{theorem}

Given the above theorem we can equivalently define the shredded sphere
$ \shred_\alpha$ as the quotient metric space $ [0,1]/ \{
\mathbf{D}^{\ast}=0\}$ equipped with the (projection of) $
\mathbf{D}^{\ast}$ or similarly using $ \mathbf{V}$ instead of $
\mathbf{D}^\ast$. Let us show why this, combined with our asymptotic
upper and lower bounds for the distances in $\randdual$,  will easily
imply Theorem \ref{thm:conv}. \label{sec:defshred} 

\begin{proof}[Proof of Theorem \ref{thm:conv}] 
Recall from the discussion after Proposition \ref{prop:firstscaling} that we have assumed that $ \widetilde{\mathcal{E}}_n \to \ce$
uniformly as $n \to \infty$ almost surely. Recall also that 
$\randdual \backslash \{\varnothing^\up, \varnothing^\down\}$
(considered as a metric space)
is the
quotient of $[0,1]$ by $\{\widetilde{D}_n =0\}$ 
endowed with the
projection of $ \widetilde{D}_n$. Similarly, $ \shred_\alpha$ is the
quotient of $[0,1]$ by the pseudo-distance $ \mathbf{V}$ or
equivalently by $  \mathbf{D}^\ast$. Those projections from $[0,1]$
induce a natural correspondence between $ n^{-1/\alpha} \cdot
\randdual \backslash \{\varnothing^\up, \varnothing^\down\}$ and $
\shred_\alpha$ whose distortion converges to $0$ by Propositions
\ref{prop:upbound} and \ref{prop:lowbound} 
combined with 
Theorem \ref{thm:V=D*}. Since
$\{\varnothing^\up, \varnothing^\down\}$ are at distance $
n^{-1/\alpha}$ from another vertex in $ n^{-1/\alpha}\cdot \randdual$
the result is granted. 
\end{proof}

Let us give a rough idea of the proof of Theorem
\ref{thm:V=D*}. First,
it is straightforward to see that the inequality
$\mathbf{V} \leq\mathbf{D}^\ast$ is always valid for
any c\`adl\`ag function $\ce$.
Indeed, for any $s=a_{1},b_{1}, \dots, a_{k},b_{k},a_{k+1}=t$
one can easily construct a path $\gamma$ going from $(s, \ce(s))$ to
$(t, \ce(t))$ and not crossing $ \mathbf{Slits}(\ce)$ whose height
variation is as close as we wish to $ \sum_{i=1}^{k}
\mathbf{D}^{\up}(a_{i},b_{i})+
\mathbf{D}^{\down}(b_{i},a_{i+1})$. Hence
  \begin{eqnarray}   \label{eq:V<D*} \mathbf{D}^\ast(s,t) \geq\mathbf{V}(s,t), \qquad \forall s,t \in [0,1].   \end{eqnarray}

However the inequality $ \mathbf{D}^\ast\leq\mathbf{V}$ is much
harder.  
We give an example below which shows that this is not a
deterministic result (indeed we shall rely on probabilistic properties
of $ \ce$).  We shall first prove 
a version of the inequality $ \mathbf{D}^\ast \leq \mathbf{V}$ for
an unconditioned L\'evy process $X$.
 To do that, we
introduce an exploration method of a path not crossing 
$\mathbf{Slits}(X)$ for a given scale $ \varepsilon>0$. This encodes
the path into a word on two letters $\tth$ and $\ttb$. For each word
obtained, this decomposes $X$ into i.i.d.~pieces for which we can
estimate the height variation needed to cross it. The proof then goes
in two stages. We first prove that
$\mathbf{D}^\ast\leq C \cdot \mathbf{V}$ for some constant $C>1$;
then we bootstrap 
the argument to show that $C$ must be equal to $1$.  
This last scheme of proof  was used earlier by Gwynne and 
Miller \cite{gwynne-miller} in a
very different and independent context, also dealing with random
metrics.

\begin{remark}[A deterministic counter-example]
\label{ex:counterex}
Introduce the functions
\[
f_1(t)=\left\{\begin{array}{ll}
0, & \mbox{if } t\in[0,1/3),\\
3(1-2t), & \mbox{if } t\in[1/3,2/3),\\
0, & \mbox{if } t\in[2/3,1],
\end{array}\right.
\mbox{ and for $k\geq1$, }\quad
f_{k+1}(t)=\beta
\left\{\begin{array}{ll}
f_k(3t), & \mbox{if } t\in[0,1/3),\\
0, & \mbox{if } t\in[1/3,2/3),\\
f_k(3t-2), & \mbox{if } t\in[2/3,1],
\end{array}\right.
\]
where $\tfrac12\leq\beta\leq1$.   Let  (see Figure \ref{fig:counterex})
\[
F(t)=\sum_{k\geq1} f_k(t).
\]
For the function $F$ we have 
$\mathbf{V}(0,1)=0$ (for $\gamma$ one can
choose a straight horizontal line) and it is not hard to show that
 $ \mathbf{D}^\ast(0,1)=2$.   
\begin{figure}[hbt]
\centering
\includegraphics [scale=.55]{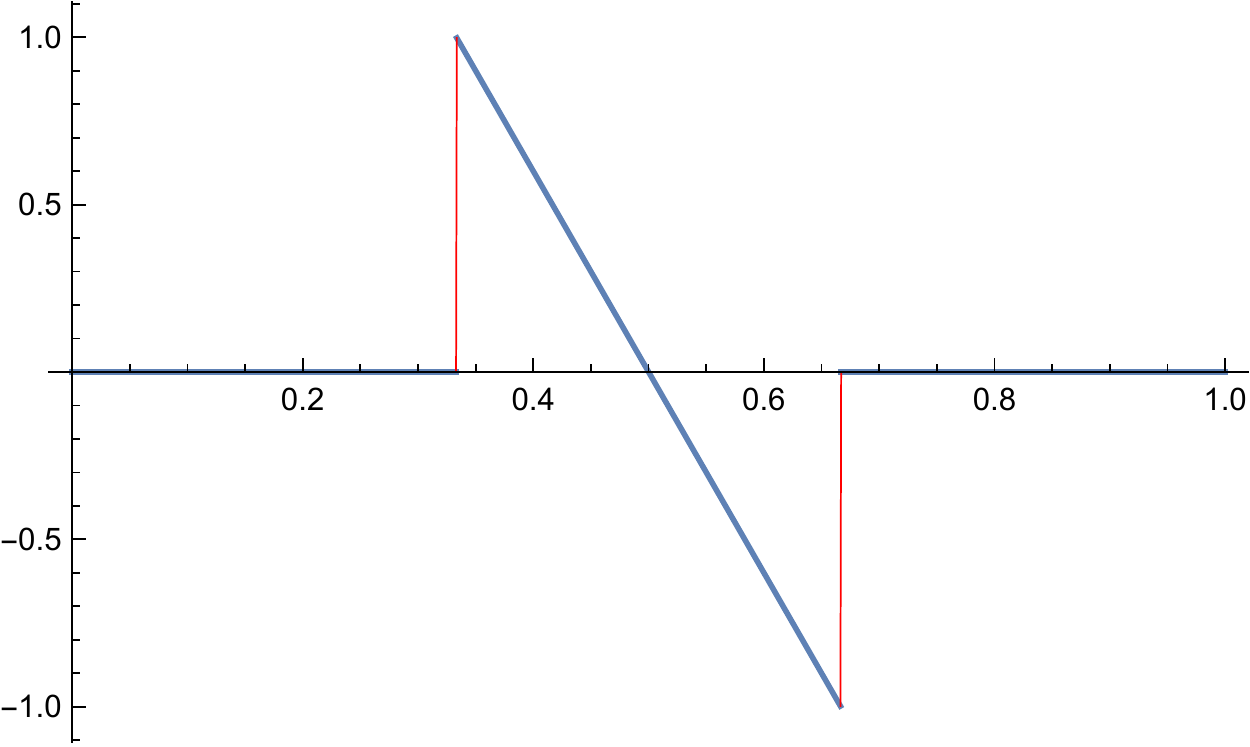}
\qquad
\includegraphics [scale=.55]{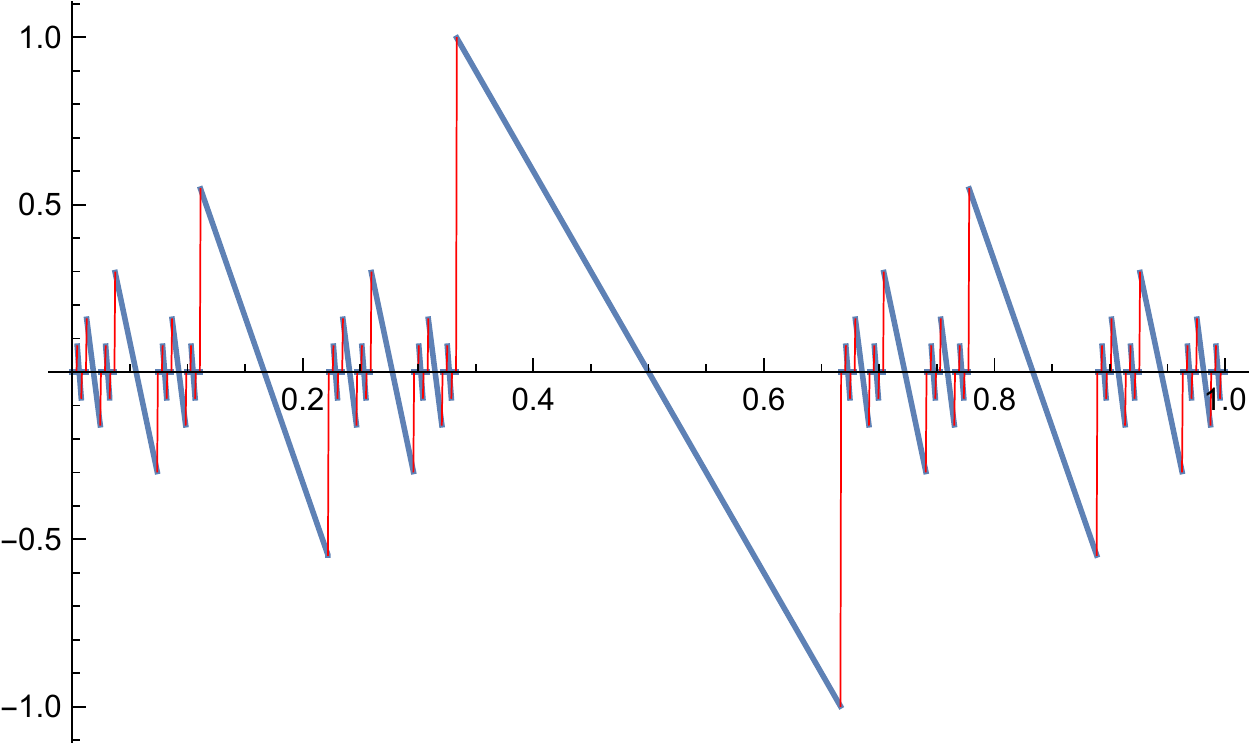}
\caption{Left: $f_1(t)$.  Right: $\sum_{k=1}^5 f_k(t)$.  The slits are
  drawn in red.  In this drawing we chose $\beta=0.55$.}
\label{fig:counterex}
\end{figure}
\end{remark}

\subsection{Exploration of a path}
We now turn to the proof of Theorem \ref{thm:V=D*}. We shall first
prove the result for the (unconditioned) $\alpha$-stable L\'evy process
$X=(X_t)_{t\geq0}$, the result for the excursion $ \ce$ will follow by  local 
absolute continuity. 
The main idea is an encoding of a path $\gamma$
at a given scale by a word which induces a Markovian exploration of the
L\'evy process $X$.

We define the corresponding objects 
$\mathbf{V},\mathbf{D}^\ast, \mathbf{Slits},\dotsc$ for $X$ as
follows.  Firstly, for $0\leq s<t$,
\[
\mathbf{D}^{\up}_X(s,t) = X(s)+ X(t) - 
2 \inf_{u \in [s , t]} X(u),\qquad
\mathbf{D}^{\down}_X(s,t) = 
2 \sup_{u \in [s , t]} X(u)  -X(s)- X(t) 
\]
and
\[
 \mathbf{D}^{\ast}_X(s,t) = 
\inf \left\{ \sum_{i=1}^{k}  \mathbf{D}_X^{\up}(a_{i},b_{i})+  
\mathbf{D}_X^{\down}(b_{i},a_{i+1}): 
{s=a_{1}, b_{1},  \dotsc, a_{k}, b_{k}, a_{k+1}=t}\right\}.
\]
Next,
$\mathbf{Slits}(X) = \bigcup_{i\geq 0} \{t_i\} \times [ X(t_i-) ,
X(t_i)]$
where $(t_i : i \geq 0)$ is an enumeration of the jumps of $X$, and
$\mathbf{V}(s,t) = \inf\mathrm{VarHt}(\gamma)$,
where the infimum is taken over all continuous paths $\gamma$ on 
$[0,\infty) \times \mathbb{R}$ going from 
$(s, X(s))$ to $(t, X(t))$ and which 
do not cross any segments of $\mathbf{Slits}(X)$.
Note that we do not use periodic time for these definitions.
By abuse of notation we still write 
$\mathbf{V},\mathbf{D}^\ast, \mathbf{Slits},\dotsc$ for 
$\mathbf{V}_X,\mathbf{D}^\ast_X, \mathbf{Slits}(X),\dotsc$ and trust
that the choice will be clear from context.

\subsubsection{Exploration along a fixed word at scale 
$\varepsilon>0$} 
For each positive integer  $\ell$ we will consider
`words' $w\in\{\ttb,\tth\}^\ell$ of length $\ell$ 
on the two letters $\ttb$ and $\tth$.  We write
$\#\tth(w)$ and $\#\ttb(w)$ for the number of entries $\tth$ 
and $\ttb$ in a word $w$, respectively, 
and $\#w=\#\tth(w)+\#\ttb(w)=\ell$ for its total length\footnote{Since
  the letters $ \mathrm{u}$ for ``up'' and $ \mathrm{d}$ for ``down''
  are already used, $ \ttb$ stands for the French ``bas'' (down) and
  $\tth$ for ``haut'' (up)}.  
For $\eps>0$ and a word of length $\ell$ define a 
sequences of stopping times $(S_k,T_k)_{k=1}^\ell$,
 as follows. First $S_0=T_0=0$, and then for $1 \leq k \leq \ell$,
\begin{align}
  S_{k}&=\inf\{t\geq T_{k-1}: X_t\geq X_{T_{k-1}}+\eps\}\\
  T_{k}&=\left\{
  \begin{array}{ll}
    S_{k}, & \text{if } w_k=\tth,\\
    \inf\{t\geq S_{k}: X_t=X_{S_{k}-}\}, &\text{otherwise}.
  \end{array}
  \right.
\end{align}
See Figure \ref{fig:word}.  Let us describe by a few sentences the
definition of these stopping times. Initially, we sit on the point
$(0,0)$ of $X$. We first wait for the first time $S_{1}$ where the
process $X$ exceeds level $ \varepsilon>0$. This will happen in finite
time and almost surely the process $X$ makes a jump at that time. We
then look at the first letter of our word $w$: if it is an $\tth$,
then we decide to go ``above'' that jump and we move on to the point
$(S_{1},X_{S_{1}})$. If it is a $\ttb$, then we decide to go ``below''
that jump: we then first wait until $T_{1}$ for the L\'evy process to
go down and reach again the value $X_{S_{1}-}$ and then we move on to
the point $(T_{1}, X_{T_{1}})$. Notice that $X$ is continuous at
  time $T_{1}$ since it has no negative jumps, in particular
  $X_{T_{1}}= X_{T_{1}-}=X_{S_{1}-}$.  We then iterate that
construction and by  induction all those stopping times are almost
surely finite and $X$ makes a jump at each time $S_{k}$. 

\begin{figure}[htb]
\centering
\includegraphics[scale=1]{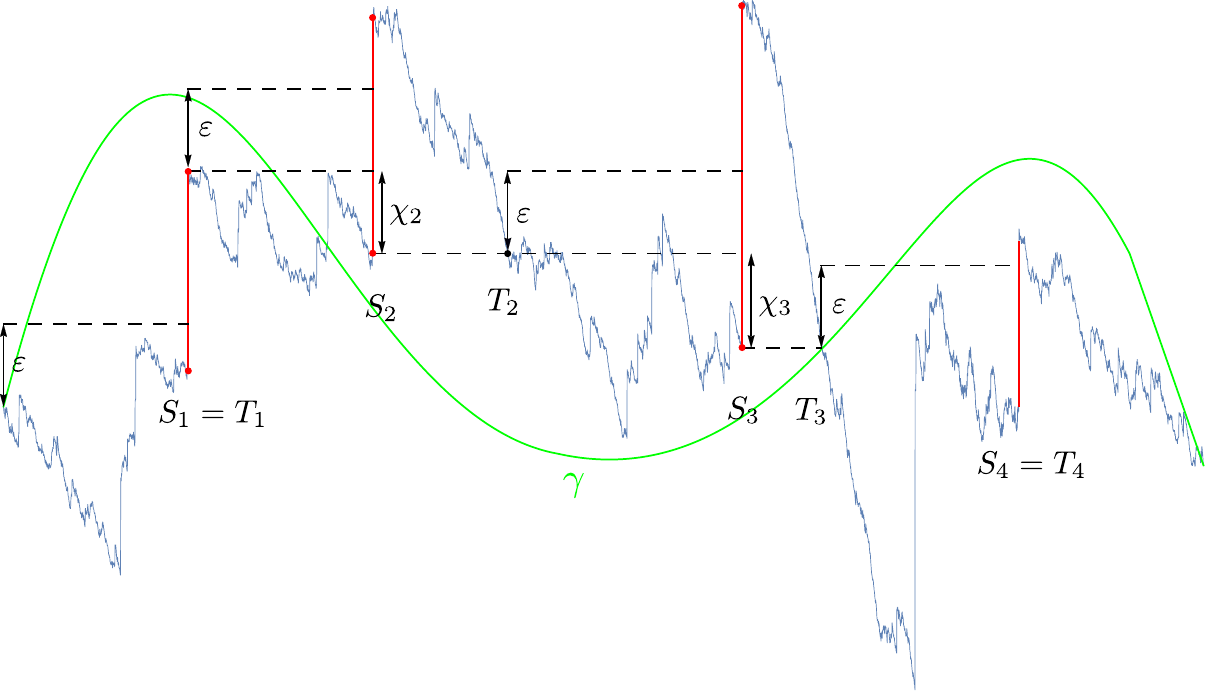}
\caption{Example of a curve $\gamma$ (green) with corresponding word 
$W^\eps(\gamma)=\tth\ttb\ttb\tth$.}
\label{fig:word}
\end{figure} 

To each $k$ such that $w_k=\ttb$, we associate a random
variable
\begin{equation}
\chi_k=\tfrac1\eps (X_{T_{k-1}}-X_{S_{k}-}),
\end{equation}
which measures how far `below' $X_{T_{k-1}}$ we find
the bottom of the jump at time $S_{k}$.
The values of $\chi_k$ are in the range $[-1,\infty)$
and
the distribution of $\chi_k$ does not depend on $\eps$. Note that for a fixed word $w$, by the strong Markov property the processes
$(X_{t+T_k}-X_{T_k})_{0\leq t\leq S_{k+1}-T_k}$ 
for $k=0,\dotsc,\ell-1$ are i.i.d.
In particular, the $\chi_k$ are i.i.d. of law $\chi$.

In the next subsection we shall derive estimates valid for a fixed word $w$ but we shall later apply our results to random words obtained by exploring a path $\gamma$ at scale $ \varepsilon>0$. To get some intuition for the coming lemmas let us describe this construction right now.

\subsubsection{Word associated to a path at scale $ \varepsilon>0$} 
Fix $t\geq 0$. We consider a continuous path
$\gamma=(\gamma_x, \gamma_y)$  going from $(0,X_0)$ to $(t,X_t)$, 
with $\gamma_x$ non-decreasing, 
which does not cross $\mathbf{Slits}(X)$. For any $ \varepsilon>0$, we associate with this path and the scale $ \varepsilon>0$, a word 
$$W^\eps (\gamma)$$ on the letters $\{ \ttb, \tth\}$ as follows.
Build the word $W^\eps(\gamma)$ letter-by-letter:  we first wait for
time $S_{1}$ when $X$ exceeds level $ \varepsilon>0$. Since the path
$\gamma$ cannot cross the slit caused by the jump of $X$ at time
$S_{1}$ it must get around either above or below it. If it goes above,
then the first letter of $W^\eps (\gamma)$ is set to $\tth$; 
otherwise it is set to $\ttb$. In the second case, we wait for the
L\'evy process $X$ to come back to the point $X_{S_{1}-}$ and we
iterate the construction. The words stop after $\ell$ letters where 
$\ell = \inf\{k \geq 0 :S_{k+1} \geq t\}$. Note
that the random  
word $W^\eps(\gamma)$ then  depends on the whole process
$(X_s)_{0\leq s\leq t}$ via the path $\gamma$. See Figure
\ref{fig:word}. The word $W^\eps(\gamma)$ encodes partial information
about $ \mathbf{D}^\ast(0,t)$ and $  \mathrm{VarHt}(\gamma)$:

  \begin{lemma} With the above notation we have \label{lem:V-bounds}
    \begin{align}
      \label{eq:Dstar-up}
       \mathbf{D}^\ast(0,t)&\leq   \mathrm{VarHt}(\gamma) +2\eps\; \#\ttb(W^\eps(\gamma)) + 2 \varepsilon, \\
      \label{eq:V-lb}
      \mathrm{VarHt}(\gamma) &\geq \eps\; \#\tth(W^\eps(\gamma))+
          \eps\sum_{k:W_k^\eps(\gamma)=\ttb} \chi_k.
    \end{align}
    In particular, when $\gamma$ is a path minimizing $
    \mathbf{V}(0,t)$ as in Lemma \ref{lem:min} we can replace $
    \mathrm{VarHt}(\gamma)$ by $\mathbf{V}(0,t)$ in the above
    displays. 
  \end{lemma}
  \begin{proof} 
We decompose $W^\varepsilon(\gamma)$ into consecutive blocks of
$\tth$'s and $\ttb$'s:
\[
W^\varepsilon(\gamma)=\tth^{i_1}\ttb^{i_2}\tth^{i_3}
\dotsb \tth^{i_{m-1}} \ttb^{i_{m}}.
\]
We set $k_0=0$, $T^\ast_0=0$, and for $1\leq j\leq m-1$ we let
$k_j=i_1+\dotsb+i_j$ 
and $T^\ast_j=T_{k_j}$, and also set $T^\ast_{m}=t$.  Thus
the    $T^\ast_j$
delimit the times (for $X$) when $W^\varepsilon(\gamma)$ switches between
$\tth$ and $\ttb$ or vice versa.  
Note that $\gamma$ starts at $(0,0)=(T^\ast_0,X(T^\ast_0))$,
then goes above $(T^\ast_1,X(T^\ast_1))$, below 
$(S^\ast_2,X(S^\ast_2-))=(S^\ast_2,X(T^\ast_2))$,
 above $(T^\ast_3,X(T^\ast_3))$, and so on, ending at 
 $(T^\ast_m,X(T^\ast_m))=(t,X(t))$.  Using continuity, this means that
 there are times 
$0=t^\ast_0\leq t^\ast_1\leq t^\ast_2\leq\dotsb\leq t^\ast_m=1$
such that $\gamma_y(t^\ast_j)=X(T^\ast_j)$ for all $j$.  
We will define another
path $\gamma^\ast$ which satisfies
$\gamma^\ast(t^\ast_j)=(T^\ast_j,X(T^\ast_j))$
for all $0\leq j\leq m$.
For both bounds \eqref{eq:Dstar-up} and \eqref{eq:V-lb}
we will then use the simple estimate
\begin{equation}\label{eq:simple}
\mathrm{VarHt}(\gamma)
\geq \sum_{j=1}^{m} 
|\gamma_y(t^\ast_{j})-\gamma_y(t^\ast_{j-1})|
= \sum_{j=1}^{m} 
|\gamma_y^\ast(t^\ast_{j})-\gamma_y^\ast(t^\ast_{j-1})|.
\end{equation}
Now we define 
$\gamma^\ast(s)$ for $t^\ast_{j-1}\leq s\leq t^\ast_j$.  
The construction is best understood with a picture, see
    Figure~\ref{fig:geo}.   
\begin{itemize}[leftmargin=*]
\item For odd $j$, meaning that we are in a block of $\tth$'s, we let
  $\gamma^\ast$ iteratively
go first up by $\varepsilon$ from $(T_k,X(T_k))$, then
  across to the next slit at time $S_{k+1}=T_{k+1}$, then up along that slit to 
$(T_{k+1},X(T_{k+1}))$, and so on,
as on the left in Figure \ref{fig:geo}.
If the block of $\tth$'s is at the end of $W^\varepsilon(\gamma)$
(i.e.\ $j=m-1$ and $i_m=0$) then we finish $\gamma^\ast$
from $(T_\ell,X(T_\ell))$ in a similar way, going down 
from $(t,X(T_\ell)+\varepsilon)$ at the very end.
\item For even $j$, meaning that we are in a block of $\ttb$'s, we let
 $\gamma^\ast$  iteratively go up by $ \varepsilon$ 
from $(T_k,X(T_k))$, then across until
  time $S_{k+1}$, then  down along the slit to height 
$X(S_{k+1}-)=X(T_{k+1})$, then across to $(T_{k+1},X(T_{k+1}))$,
and so on,
 as on the right in Figure \ref{fig:geo}.
If the block of $\ttb$'s is at the end of $W^\varepsilon(\gamma)$
(i.e.\ $j=m$) then we finish $\gamma^\ast$
from $(S_\ell,X(S_\ell-))$ in a similar way, either straight across
and up if $T_\ell>t$, otherwise up by $\varepsilon$ from 
$(T_\ell,X(T_\ell))$, across, and finally down.
\end{itemize}

  \begin{figure}[!h]
   \begin{center}
   \includegraphics[width=16cm]{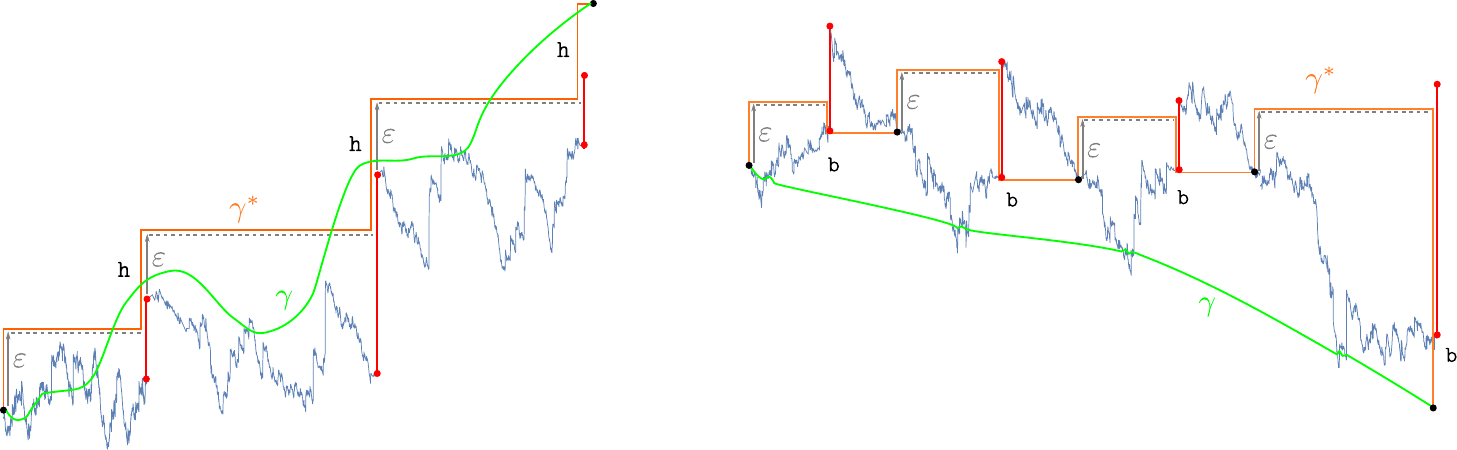}
   \caption{ \label{fig:geo} Illustration of the proof. In the case of
     a sequence of $ \tth$ steps (left), the height difference between
     the starting and the endpoint of $ \gamma$ can be matched by a
     path using the down-tree (i.e.\ $\mathbf{D}^\down$) only. In the case
     of a sequence of 
     $\ttb$ steps (right), using both $ \mathbf{D}^\up$ and $
     \mathbf{D}^\down$, we can approximate the height difference up to
     an error of $2 \varepsilon$ per letter $\ttb$.} 
   \end{center}
   \end{figure}

We now note the following about the height-variations 
$\mathrm{VarHt}(\gamma^\ast(s):t^\ast_{j-1}\leq s\leq t^\ast_j)$.
For simplicity, let us assume that $T^\ast_j<t$ (similar considerations
apply to the last bit of path).
\begin{itemize}[leftmargin=*]
\item For odd $j$, $\gamma^\ast_y(s)$ is non-decreasing in the
  interval, and we have
\begin{equation}\label{eq:useful1}
\mathrm{VarHt}(\gamma^\ast(s):t^\ast_{j-1}\leq s\leq t^\ast_j)
= \gamma_y^\ast(t^\ast_{j})-\gamma_y^\ast(t^\ast_{j-1})
\geq \varepsilon i_j.
\end{equation} 
Actually, the height-variation of
$\gamma^\ast$ realises the $\mathbf{D}^\down$-distance we
  between the endpoints:
\begin{equation}\label{eq:useful2}
\mathrm{VarHt}(\gamma^\ast(s):t^\ast_{j-1}\leq s\leq t^\ast_j)
=\mathbf{D}^\down(T^\ast_{j-1},T^\ast_j).
\end{equation} 
\item For even $j$, by summing the sizes of the vertical steps we get 
\begin{equation}\label{eq:useful3}
\mathrm{VarHt}(\gamma^\ast(s):t^\ast_{j-1}\leq s\leq t^\ast_j)
=\sum_{k=k_{j-1}}^{k_j-1} (2\varepsilon+\varepsilon \chi_k).
\end{equation}
 Summing instead the displacements  (with  sign), we get 
\begin{equation}\label{eq:useful4}
\gamma_y^\ast(t^\ast_{j})-\gamma_y^\ast(t^\ast_{j-1})
=-\sum_{k=k_{j-1}}^{k_j-1} \varepsilon \chi_k.
\end{equation}
Combined with \eqref{eq:useful3}, this gives
\begin{equation}\label{eq:useful5}
\mathrm{VarHt}(\gamma^\ast(s):t^\ast_{j-1}\leq s\leq t^\ast_j)\leq 
|\gamma_y^\ast(t^\ast_{j})-\gamma_y^\ast(t^\ast_{j-1})|
+2\varepsilon i_j.
\end{equation}
The way we chose $\gamma^\ast$  means that
 we also have 
\begin{equation}\label{eq:useful6}
\mathrm{VarHt}(\gamma^\ast(s):t^\ast_{j-1}\leq s\leq t^\ast_j)
\geq \sum_{k=k_{j-1}}^{k_j-1} \big[
\mathbf{D}^\down(T_{k},S_{k+1})+
\mathbf{D}^\up(S_{k+1},T_{k+1})\big].
\end{equation}
\end{itemize}

To prove \eqref{eq:Dstar-up}, we use the decomposition of $[0,t]$
suggested by \eqref{eq:useful2} and \eqref{eq:useful6}, the
identity in
\eqref{eq:useful1}, and the bound in \eqref{eq:useful5} to get
\[
\mathbf{D}^\ast(0,t)\leq \sum_{j=1}^{m} \big[
|\gamma_y^\ast(t^\ast_{j})-\gamma_y^\ast(t^\ast_{j-1})|+
2\varepsilon i_j 1\{j\mbox{ even}\}\big]+
2\varepsilon,
\]
where the last $2\varepsilon$ comes from the last bit of
$\gamma^\ast$.  
Now \eqref{eq:Dstar-up} follows from
 \eqref{eq:simple}.
Finally, \eqref{eq:V-lb} follows by using 
\eqref{eq:simple}
together with the bound in
\eqref{eq:useful1} and the identity in \eqref{eq:useful4}.
\end{proof}

\subsection{Probabilistic estimates for fixed words}

In this section we fix a (long) word $w$ and estimate its height
difference $\sum \chi_k$. Our goal is to get large deviations
estimates in order to prove a statement valid for \emph{all} words
simultaneously (in particular  to words $W^\eps(\gamma)$ coming from
the exploration of \emph{any} path at scale $ \varepsilon>0$). 
Recall that $\alpha \in (1,2)$.

We start by an estimate on the `underjump'  of law $\chi{=}\chi_k$, which by a standard result in renewal theory
is related to the size-biasing of the L\'evy measure. In particular $
\mathbb{P}( \chi \geq r)$ decays as $r^{-\alpha+1}$ as $r \to
\infty$. For our purposes we shall need:  
\begin{proposition}\label{prop:chi-dom}
Write $\beta=\alpha-1\in(0,1)$ and let $ \mathfrak{S}$ be a strictly positive stable random variable of index $\beta$, i.e. $ \mathbb{E}[ \mathrm{e}^{-\lambda \mathfrak{S}}] = e^{-\lambda ^\beta}$ for all $\lambda >0$. Then there are constants $a,b>0$ depending only on $\alpha$ such that we have the stochastic domination 
$$ -a + b \ \mathfrak{S} \leq \chi.$$
\end{proposition}
\begin{proof}
  The density of $\chi$ is explicitly known, see
  \cite[Example 7]{doney-kyprianou} ($\chi$ has the distribution of
  $-X_{\tau_1^+-}$ in their notation).  In particular, there is some
  constant $c$ such that for all $r\geq 0$ we have
$\P(\chi>r)=c(r+1)^{-\beta}$.  Also, there is another constant $C>0$
such that $\P(\mathfrak{S}>r)\leq C r^{-\beta}$ for all $r\geq0$,
see \cite[Property~1.2.15]{stable-1994}.
Thus  we can pick $a,b>0$ as claimed.
\end{proof}

\begin{proposition}\label{prop:chi-sum}
  There is a constant $c_\alpha$ depending only on $\alpha$ such that
  the following holds.  
  Let $w\in\{\ttb,\tth\}^\ell$ be  a fixed word with
  $\#\ttb=m$.   Then for any $q>0$, writing 
$\kappa=\tfrac{\alpha-1}{2-\alpha}>0$, we have that
  \begin{equation}
    \P\Big(\sum_{k:w_k=\ttb} \chi_k\leq qm\Big)
    \leq \exp(-c_\alpha q^{-\kappa}m).
  \end{equation}
\end{proposition}
\begin{proof}
  Writing $\mathfrak{S}(t)$ for a $\beta=(\alpha-1)$-stable
  subordinator viewed at 
  time $t$,
  and $\mathfrak{S}_k$ for independent copies of
  $\mathfrak{S}(1)\overset{(d)}{=}\mathfrak{S}$,
  we have using Proposition \ref{prop:chi-dom} and the fact that the $(\chi_{k}~:~w_k = \ttb)$ are i.i.d.~for any given word $w$ that 
\begin{equation}\begin{split}
  \P\Big(\sum_{k:w_k=\ttb} \chi_k\leq qm\Big)
  &\leq \P\Big(\sum_{k:w_k=\ttb}\mathfrak{S}_k\leq \tfrac{q+a}b m\Big)
  = \P\big(\mathfrak{S}(m)\leq \tfrac{q+a}b m\big)\\
    &\underset{ \mathrm{scaling}}{=} \P\big(\mathfrak{S}(1)\leq \tfrac{q+a}b qm^{1-1/\beta}\big).
  \end{split}
\end{equation}
For $\delta\in(0,1)$ we have \cite[p. 221]{bertoin:1996}
 that
\begin{equation}
  \P(\mathfrak{S}(1)\leq \delta)\leq \exp(-c
  \delta^{\beta/(\beta-1)})
  \end{equation}
  for some constant $c$ depending only on $\alpha$.
  We get
  \begin{equation}
    \P\big(\mathfrak{S}(1)\leq \tfrac{q+a}b m^{1-1/\beta}\big)
    \leq \exp(-c(\tfrac{q+a}b)^{\beta/(\beta-1)}m)
  \end{equation}
  as claimed.
\end{proof}

\subsection{A first bound}
In this section we use the previous estimates to show a first bound: There exists \emph{some} $\delta>0$ such that almost surely, in the L\'evy process $X$ we have
  \begin{eqnarray} \label{eq:goalfirst}  \forall t\geq 0, \quad  \mathbf{V}(0,t) \leq \mathbf{D}^\ast(0,t) \leq (1 + \delta) \mathbf{V}(0,t). \end{eqnarray}

  We will eventually see that $\delta$ can be made arbitrarily small, however in this first step we will need to think about $\delta$ as a large constant.
  Recall
  that the inequality $ \mathbf{V} \leq \mathbf{D}^\ast$ is
  deterministically true by extending \eqref{eq:V<D*} from  
the case of $ \mathcal{E}$ to the case of $X$. 
Towards \eqref{eq:goalfirst}, fix $t >0$ and
  consider a continuous path $\gamma$ starting from $(0,0)$ and ending
  at $(t, X_t)$ minimizing the height variation as in Lemma
  \ref{lem:min}. We shall consider the associated word
  $W^\eps(\gamma)$ as $ \varepsilon \to 0$. Let us give first the
  rough idea of the proof to help the reader follow our
  steps. First, if the word is short, in the sense that $\#
  W^{\eps}(\gamma) =o(1/ \varepsilon)$ then \eqref{eq:Dstar-up}
  automatically gives $ \mathbf{D}^{*}(0,t) = \mathbf{V}(0,t)$. The
  problem might come from long words $W^{\eps}( \gamma)$ containing
  more that $ \varepsilon^{-1}$ letters, i.e. of paths $\gamma$
  oscillating frenetically around our L\'evy process $X$. We may seem
  in bad shape since there are $2^{ \eps^{-1}}$ words of length $
  \varepsilon^{-1}$. Our salvation will come from the large deviations
  bound in Proposition \ref{prop:chi-sum}  which roughly says that if
  the word has too many letters $\ttb$, then the path $\gamma$ has a
  large height variation. While the exponential control in Proposition
  \ref{prop:chi-sum} does not readily ``kill'' the entropic term $2^{
    \eps^{-1}}$,  it does so if we restrict to words with a small
  proportion of letters $\tth$. 

Let us proceed. Using the
  deterministic 
  bounds of Lemma \ref{lem:V-bounds} we can prove: 
  \begin{proposition} \label{prop:bad}
    Let $t >0$ be such that $ \mathbf{D}^\ast(0,t) > (1+ \delta)
    \mathbf{V}(0,t)$. Then eventually as $ \varepsilon \searrow 0$ we 
    have that $  \#  \ttb(W^\eps(\gamma)) \geq \varepsilon^{-1/2}$ and
    that $W^\eps(\gamma)$ belongs to the set  
    \[
      \mathrm{Bad}= \left\{ \mbox{words }w : \#\tth(w)\leq  \frac{3}{\delta} \#w, 
        \sum_{k:w_k=\ttb}\chi_k\leq  \frac{3}{\delta} \#\ttb(w)\right\}.
      \]
\end{proposition} 
\begin{proof}
  The first assertion is easy to prove.
  Indeed if the number of $\ttb$ letters in $W^\eps(\gamma)$ is 
$<\eps^{-1/2}$ infinitely often as $\eps\to0$
  then  \eqref{eq:Dstar-up} 
gives $ \mathbf{D}^\ast(0,t)\leq  \mathbf{V}(0,t)+4\eps^{1/2}$
meaning that $ \mathbf{D}^\ast(0,t)= \mathbf{V}(0,t)$. 
For the second assertion,  if $t>0$ is such that $ \mathbf{D}^\ast(0,t) > (1+ \delta) \mathbf{V}(0,t)$  then \eqref{eq:Dstar-up} 
gives for any $\eps>0$ with $W= W^\eps(\gamma)$
\begin{equation} \label{eq:11}
\mathbf{V}(0,t)+2\eps\#\ttb(W) + 2 \varepsilon \underset{\eqref{eq:Dstar-up}}{\geq}  \mathbf{D}^\ast(0,t) \underset{ \mathrm{assumpt.}}{>}(1+\delta) \mathbf{V}(0,t).
\end{equation}
Consequently $\#\ttb(W)> \tfrac\delta2\eps^{-1}  \mathbf{V}(0,t)- 1$.
At the same time, \eqref{eq:V-lb} gives 
$\#\tth(W)\leq  \eps^{-1}   \mathbf{V}(0,t)$.  Putting these together we conclude
that eventually we have
$\#\tth(W)\leq \tfrac3\delta\#\ttb(W) \leq\tfrac3\delta \#W$.
Moreover, using \eqref{eq:V-lb}  again, we also have
\begin{equation}
\eps\sum_{k:W_k=\ttb} \chi_k \underset{\eqref{eq:V-lb}}{\leq}  \mathbf{V}(0,t)  \underset{ \eqref{eq:11}}{<} \frac{2 \varepsilon}{\delta}(\#\ttb(W) +1),
\end{equation}
and thus eventually we have $\sum_{k:w_k=\ttb}\chi_k\leq
\frac{3}{\delta} \#\ttb(W)$, using the fact that $\#\ttb(W) \to \infty $ by the first
part of the proof.
\end{proof}

Given the previous lemma, the next result finishes the proof of
\eqref{eq:goalfirst}. 
\begin{lemma} \label{lem:BNpq}
  For $\delta >0$ large enough, with probability one, the set of words
  of length $\ell$ belonging to $ \mathrm{Bad}$ is eventually empty as
  $\ell\to\infty$. 
\end{lemma}
\begin{proof}
For $\ell \geq 1$, $p\in (0,\tfrac12)$ and $q>0$ define 
\[ 
\mathcal{B}(\ell,p,q)=
\left\{ \mbox{ words }w : \#w = \ell, \#\tth(w)\leq p \#w, 
\sum_{k:w_k=\ttb}\chi_k\leq q \#\ttb(w)\right\}.
\]
This is a random set of words because of the variables $\chi_{k}$
which depend on $X$. 
It is well-known that the number of words $w$ satisfying 
$\#w = \ell$ and $\#\tth(w)\leq p \#w$ is at most 
$\exp(\ell[p\log(\tfrac1p)+(1-p)\log(\tfrac1{1-p})])$,
see e.g.~\cite[Theorem~3.1]{galvin14}.
Summing over all such words $w$, we
can use Proposition \ref{prop:chi-sum} to  bound
\begin{equation}\begin{split}
\E[\#\mathcal{B}(\ell,p,q)]&\leq 
  \sum_{\substack{w:\\\#\tth(w)\leq p \ell\\\#w = \ell}}
  \P\Big(\sum_{k: w_k=\ttb} \chi_k\leq q \#\ttb(w)\Big)\\
&\underset{ \mathrm{Prop.} \ref{prop:chi-sum}}{\leq} \sum_{\substack{w:\\\#\tth(w)\leq p \ell\\\#w = \ell}}
\exp\big(-c_\alpha q^{-\kappa}\#\ttb(w)\big)\\
& \leq
\exp\big(-\ell[c_\alpha q^{-\kappa}(1-p)
-p\log(\tfrac1p)-(1-p)\log(\tfrac1{1-p})]\big)
\end{split}
\end{equation}
where 
$c_\alpha,\kappa>0$ are from Proposition \ref{prop:chi-sum}. Now,
putting $p=q = \frac{3}{\delta}$ in the last display, we can choose
$\delta$ large enough so that $c_\alpha q^{-\kappa}(1-p)
-p\log(\tfrac1p)-(1-p)\log(\tfrac1{1-p}) >0$. 
Then the series 
$\sum_\ell \E[\#\mathcal{B}(\ell,\tfrac3\delta,\tfrac3\delta)]$ 
is convergent, and thus $\sum_\ell\#\mathcal{B}(\ell,\tfrac3\delta,\tfrac3\delta)$ is convergent almost surely from which it follows that 
$\mathcal{B}(\ell, \tfrac3\delta, \tfrac3\delta)$ 
is eventually empty almost surely. 
\end{proof}

\subsection{Bootstrapping}
In this section we finally prove $ \mathbf{D}^\ast = \mathbf{V}$ for the L\'evy process $X$. Our strategy is first to extend \eqref{eq:goalfirst} to all pairs of time $s,t \geq 0$. We then use these bounds to sharpen Lemma \ref{lem:V-bounds}, and by using the same kind of arguments as in the last section this gives $\delta=0$. Let us proceed.\medskip

We first claim that  in the L\'evy process $X$,  we have
  \begin{equation} \label{eq:goalfirstbis}   
\forall s,t \geq 0, \quad \mathbf{V}(s,t) \leq 
\mathbf{D}^\ast(s,t) \leq (1 + \delta) \mathbf{V}(s,t). 
\end{equation}
  This is true with probability one for any fixed $s=s_0 >0$ and any
  $t >s_0$ by \eqref{eq:goalfirst} and invariance by time translation. By countable intersection, the above display is true with
  probability one for any $0 \leq s \leq t$ with $s \in
  \mathbb{Q}$. The bound is then extended to all pairs $ 0 \leq s \leq
  t$ using the fact that both $s \mapsto  \mathbf{V}(s,t_0)$ and 
$s \mapsto  \mathbf{D}^\ast(s,t_0)$ are (almost surely) right-continuous in $s\leq t_0$.

We now use \eqref{eq:goalfirstbis} as a `deterministic input' in order
to sharpen the geometric estimates of Lemma \ref{lem:V-bounds}. 
Thus we suppose that we are given a right-continuous path $X$
satisfying \eqref{eq:goalfirstbis} and we aim at controlling the height
variation and the distance $ \mathbf{D}^\ast$ of a path $\gamma$
starting from $(0,0)$ and ending at $(t,X_t)$.
We assume that $\gamma$ is optimal, i.e.\  
$\mathrm{VarHt(\gamma)}=\mathbf{V}(0,t)$, as in 
Lemma \ref{lem:min}.
As in Lemma
\ref{lem:V-bounds} we consider the word $W^\eps(\gamma)$ and examine
the contributions to the height variation of $\gamma$ along sequences
of the same letter $ \ttb$ or $\tth$.  
We will use several definitions from the proof of that lemma.
Recall in particular the times 
$0=T^\ast_0<T^\ast_1<\dotsc<T^\ast_{m-1}<T^\ast_m=t$
which delimit the consecutive sequences of $\tth$'s or $\ttb$'s, 
the times $t^\ast_j$ such that 
$\gamma_y(t^\ast_j)=X(T^\ast_j)$, 
and the path $\gamma^\ast$ satisfying
$\gamma^\ast(t^\ast_j)=(T^\ast_j,X(T^\ast_j))$.
We define
\[
\mathrm{VarHt}_{\tth}(\gamma, \eps)=
\sum_{\substack{1\leq j\leq m-1\\j\text{ odd}}}
\mathrm{VarHt}(\gamma(s):t^\ast_{j-1}\leq s\leq t^\ast_j ),\quad
\mathrm{VarHt}_{\ttb}(\gamma, \eps)=
\sum_{\substack{1\leq j\leq m\\j\text{ even}}}
\mathrm{VarHt}(\gamma(s):t^\ast_{j-1}\leq s\leq t^\ast_j ),
\]
noting that the latter also includes the ``last bit of path''.  
We have
\begin{equation} \label{eq:bd1} 
\mathbf{V}(0,t)=
\mathrm{VarHt}(\gamma)=
\mathrm{VarHt}_{\tth}(\gamma, \eps)
+\mathrm{VarHt}_{\ttb}(\gamma,\eps).
\end{equation}

\begin{lemma} \label{lem:V-boundsimproved}
 We have 
 \begin{align} 
& \label{eq:bd2} \mathrm{VarHt}_{\tth}(\gamma, \eps) \geq 
\varepsilon \#\tth( W^\eps(\gamma)) \\
& \label{eq:bd3} \mathbf{D}^\ast(0,t) \leq \mathrm{VarHt}_{\tth}(\gamma, \eps) + (1+ \delta) \mathrm{VarHt}_{\ttb}(\gamma, \eps). \end{align}
\end{lemma}
\begin{proof}
For \eqref{eq:bd2}, note that by \eqref{eq:useful1} we have
\[
\mathrm{VarHt}_{\tth}(\gamma, \eps)\geq
\sum_{\substack{1\leq j\leq m-1\\j\text{ odd}}}
|\gamma_y(t_j^\ast)-\gamma_y(t_{j-1}^\ast) |
=\sum_{\substack{1\leq j\leq m-1\\j\text{ odd}}}
|\gamma_y^\ast(t_j^\ast)-\gamma_y^\ast(t_{j-1}^\ast) |
\geq \varepsilon \#\tth( W^\eps(\gamma)).
\]
For \eqref{eq:bd3}, by using subadditivity, \eqref{eq:useful1},
\eqref{eq:useful2}, 
as well as \eqref{eq:goalfirstbis}   we get
\[
\mathbf{D}^\ast(0,t)\leq 
\sum_{\substack{1\leq j\leq m-1\\j\text{ odd}}}
\mathbf{D}^\down(T^\ast_{j-1},T^\ast_j)
+
\sum_{\substack{1\leq j\leq m\\j\text{ even}}}
\mathbf{D}^\ast(T^\ast_{j-1},T^\ast_j)\leq
\mathrm{VarHt}_{\tth}(\gamma, \eps)+
(1+\delta) \sum_{\substack{1\leq j\leq m\\j\text{ even}}}
\mathbf{V}(T^\ast_{j-1},T^\ast_j).
\]
Clearly $\mathbf{V}(T^\ast_{j-1},T^\ast_j)\leq
\mathrm{VarHt}(\gamma(s):t^\ast_{j-1}\leq s\leq t^\ast_j )$,
which finishes the proof.
\end{proof}

With these improved geometric controls at hands we can now 
prove that we may take 
$\delta=0$ in \eqref{eq:goalfirstbis}. The idea is to show that if
$\delta >0$ then one can find $0<\sigma<\delta$ so that
\eqref{eq:goalfirstbis} holds with this $\delta$ replaced by
$\sigma$. By considering the smallest $\delta \geq 0$ so that
\eqref{eq:goalfirstbis} holds with probability $1$ we deduce indeed
that $\delta=0$.  

\begin{proposition}
  Suppose that \eqref{eq:goalfirstbis} holds with probability $1$ with
  some $\delta >0$. Then we can find $0<\sigma<\delta$ so that
  \eqref{eq:goalfirstbis} holds with probability $1$ with $\delta$
  replaced by $\sigma$. 
\end{proposition}
\begin{proof} By the argument presented at the beginning of this
  subsection, it is sufficient to find $0< \sigma<\delta$ so that
  \eqref{eq:goalfirst} holds almost surely when $\delta$ is replaced by $\sigma$.
  To do this, let us  set
  \[
    \sigma=\tfrac\delta2\vee(\delta(1-p\tfrac\delta4))\ \ \in
    (0,\delta),
    \]
  where $p\in(0,\tfrac12)$ is to be chosen at the end of the
  proof. Suppose $t >0$ is such that
  \begin{eqnarray} \label{eq:badt}
    \mathbf{D}^\ast(0,t) > (1+ \sigma) \mathbf{V}(0,t),
  \end{eqnarray}
                and let $\gamma$ be an optimal path as in Lemma
\ref{lem:min} going from $(0,0)$ to $(t,X_t)$. We consider the word $
W^\eps(\gamma)$ coming from the exploration of $\gamma$ at scale $
\eps$ and write $W = W^\eps(\gamma)$ to simplify notation. Let us make
the following observations on the behavior of such words. 
\begin{itemize}[leftmargin=*]
\item 
We first note that $\#W \geq \ttb(W) \geq \eps^{-1/2}$ 
eventually;  otherwise we
would have $ \mathbf{V}(0,t) = \mathbf{D}^\ast(0,t)$  as already 
observed in Proposition \ref{prop:bad}.
\item
 Our second claim is that eventually $\#\tth(W) \leq p\# W$.
To see this, we begin by noting  that
if $\#W<\tfrac\delta 4\eps^{-1}V(0,t)$
then from \eqref{eq:Dstar-up} we already have
\[
 \mathbf{D}^\ast(0,t)\underset{\eqref{eq:Dstar-up}}{\leq}
 \mathbf{V}(0,t) + 2\eps \tfrac\delta4\eps^{-1} \mathbf{V}(0,t)  + 2
 \varepsilon
 \leq  (1+ \tfrac{\delta}{2})\mathbf{V}(0,t)+2 \varepsilon
 \leq (1+\sigma) \mathbf{V}(0,t),
\]
after letting $ \varepsilon \to 0$.
We therefore proceed under the assumption
$\# W\geq \tfrac\delta 4\eps^{-1} \mathbf{V}(0,t)$, and assume by contradiction that  $\#\tth (W)\geq p \# W$. 
This gives $\#\tth (W) \geq p\tfrac\delta4 \eps^{-1}  \mathbf{V}(0,t)$ and by \eqref{eq:bd2} we have 
$$   \mathrm{VarHt}_{\tth}(\gamma, \varepsilon) \geq  p\tfrac\delta4  \mathbf{V}(0,t).$$
We can thus write 
 \begin{eqnarray*} \mathbf{D}^\ast(0,t) 
&\underset{\eqref{eq:bd3}}{\leq}&   
(1+\delta)\mathrm{VarHt}_{\ttb}(\gamma,\eps)
+\mathrm{VarHt}_{\tth}(\gamma,\eps) 
 \underset{\eqref{eq:bd1}}{=} 
\mathbf{V}(0,t)+\delta \mathrm{VarHt}_{\ttb}(\gamma,\eps)\\
  &{ \leq} &  (1+\delta(1-p\tfrac\delta4)) \mathbf{V}(0,t),
\end{eqnarray*}
  and this yields  a contradiction given our definition of $\sigma$. 
\item
Our final claim is that if we set $q=\tfrac4\delta$, then we must have 
$\sum_{k:W_k=\ttb}\chi_k \leq q\#\ttb(W)$.  Otherwise combining
\eqref{eq:Dstar-up}, \eqref{eq:V-lb} and letting $\eps \to 0$ would
give $\mathbf{D}^\ast(0,t)\leq (1+\tfrac2q) \mathbf{V}(0,t)=
(1+\tfrac\delta2) \mathbf{V}(0,t)  
\leq (1+\sigma) \mathbf{V}(0,t)$ which is excluded by assumption.
\end{itemize}

Gathering-up our findings, the words $W^\eps(\gamma)$ corresponding to
explorations of minimizing paths going to times $t >0$ satisfying
\eqref{eq:badt} eventually belong to the set $ \mathcal{B}(\ell, p,
\frac{4}{\delta} )$ for some $\ell \geq \varepsilon^{-1/2}$ where
those sets were defined in the proof of Lemma \ref{lem:BNpq}. Arguing
exactly as in that proof, for our fixed $q = \frac{4}{\delta}$ we can
find $p$ small enough so that $\mathcal{B}(\ell, p, \frac{4}{\delta}
)$ is eventually empty almost surely. This implies that there is no
such $t >0$ satisfying \eqref{eq:badt}. This means that
\eqref{eq:goalfirst} holds for $\delta$ replaced by $\sigma$ and we
extend  to all pairs of times $s,t \geq 0$ using right-continuity as
in the beginning of this section.
\end{proof}

We can now finish the proof of Theorem \ref{thm:V=D*}:
\begin{proof}[Proof of Theorem \ref{thm:V=D*}]
  Thanks to the previous
  sections the equality $  \mathbf{V}(s,t) = \mathbf{D}^\ast(s,t)$ is
  granted for $s,t \geq 0$ in the stable L\'evy process $X$. To extend
  it to the excursion $ \ce$ we use local absolute continuity. More
  precisely, if $[x]_1 \in [0,1)$ denotes the fractional part for $x
  \in \mathbb{R}$, for  $0 < s_0 < t_0 <1$, the processes $
  \left(\ce(s_0+u) - \ce(s_0) : 0 \leq u \leq t_0-s_0\right )$ and $
  \left(\ce([t_0+u]_1)- \ce(t_0) : 0 \leq u \leq 1-t_0+s_0\right )$
  are absolutely continuous with respect to $ (X_u : 0 \leq u \leq
  t_0-s_0)$ and $ (X_u : 0 \leq u \leq 1-t_0+s_0)$ respectively. This
  follows from the Vervaat transform relating $ \ce$ to the normalized
  bridge $X^{ \mathrm{br}}$ of $X$  and  absolute continuity relation
  between $X^{\mathrm{br}}$ and $X$ (see
\cite{bertoin:1996} Chapter VIII.3, Formula (8)). Lemma \ref{lem:min}
shows that to compute $ \mathbf{V}(s_0,t_0)$ it is sufficient to
restrict to path whose $x$-coordinate is monotone, that is either go
from $s_0$ to $t_0$ or from $s_0$ to $0$ and then from $1$ to
$t_0$. In both cases, we can compare with one of the above pieces of
the L\'evy process $X$ and we deduce from the last section that $
\mathbf{D}^\ast(s_0,t_0) = \mathbf{V}(s_0,t_0)$ almost surely in $
\ce$. The result is extended to all $s,t$ by
right-continuity. \end{proof}

\section{Properties of the shredded spheres}

In this section we establish a few basic properties of the shredded spheres. We first
show that the Hausdorff dimension of $ \shred_\alpha$ is $\alpha \in
(1,2)$ and characterize the point identifications made by $ \mathbf{V}$ or $ \mathbf{D}^{*}$ in Theorem \ref{l:vd0}. Towards understanding the topology of $ \shred_{\alpha}$, we study the graph formed by the faces of $\shred_\alpha$ which are the scaling limits of the large multimers in $
\randmult$. We show that adjacent faces exist using the decrease
points of stable processes \cite{bertoin:1994} but leave the question of the  connectedness of the graph of faces open.\medskip

In this section we see $ \shred_\alpha$ as the quotient of $[0,1]$ by the equivalence relation $ \mathbf{V}=0$ endowed with the projection of the pseudo-distance $ \mathbf{V}$ or equivalently of $  \mathbf{D}^{\ast}$. The canonical projection $ [0,1] \to \shred_\alpha$ is denoted by $\pi$.

\subsection{Hausdorff dimension}
\label{sec:dim}
\begin{proposition} \label{prop:dim} For $\alpha \in (1,2)$, almost surely the Hausdorff dimension of $ \shred_\alpha$ is $\alpha$.
\end{proposition}
\begin{proof}
Both the upper bound $\dim(\shred_\alpha)\leq\alpha$
 and the lower bound $\dim(\shred_\alpha)\geq\alpha$ 
can be proved very similarly
to the corresponding statements for random stable looptrees
\cite[Section 3.3]{curien:2014}, with only a small modification needed
for the lower bound (see also \cite{archer2019brownian}).  We give a brief outline of the argument and
indicate the necessary modifications.

For the upper bound we note that $\shred_\alpha$ can be covered by
the sets $\pi([t^{(\eps)}_i,t^{(\eps)}_{i+1}))$ where the
$t_i^{(\eps)}$ form an increasing enumeration of the times $t$ such
that $\Delta \ce_t>\eps$.
Due to the bound $ \mathbf{D}^\ast\leq  \mathbf{D}^\up$, the diameter of
$\pi([t^{(\eps)}_i,t^{(\eps)}_{i+1}))$ in $\shred_\alpha$ is at
most 
$2\sup\{|\ce(s)-\ce(t)|:s,t\in [t^{(\eps)}_i,t^{(\eps)}_{i+1})\}$.
The same calculations as in \cite[Section 3.3.1]{curien:2014} then give 
$\dim(\shred_\alpha)\leq\alpha$.

For the lower bound, as in \cite[Section 3.3.2]{curien:2014} it suffices to show
that for any $\eta>0$, almost surely
$\nu(B_r(\pi(U)))\leq r^{\alpha-\eta}$
for all $r>0$ sufficiently small, where 
\begin{itemize}
\item $U\in[0,1]$ is a uniform random variable independent of
  $\shred_\alpha$; 
\item $\nu(\cdot)$ is the push-forward of Lebesgue measure on $[0,1]$ under
  $\pi$; 
\item $B_r(\cdot)$ is the ball of radius $r$ in $\shred_\alpha$.
\end{itemize}
In place of \cite[Lemma 3.13]{curien:2014} we use the following:

\begin{lemma}\label{lem:blocking-jumps}
Fix $\eta>0$.  Almost surely, for every $\eps$ small enough, there are jump times
$S_\eps,T_\eps$ of $\mathcal{E}$ satisfying:
\begin{enumerate}
\item $S_\eps\in(U-\eps,U)$ and $T_\eps\in(U,U+\eps)$,
\item for $R=S$ or $T$,
\[
\mathcal{E}(R_\eps-)\leq \mathcal{E}(U)-\eps^{1/\alpha+\eta},
\qquad
\mathcal{E}(R_\eps)\geq \mathcal{E}(U)+\eps^{1/\alpha+\eta}.
\]
\end{enumerate}
\end{lemma}
See Figure \ref{f:jumps} for an illustration.
\begin{figure} [h]
 \centerline{\includegraphics{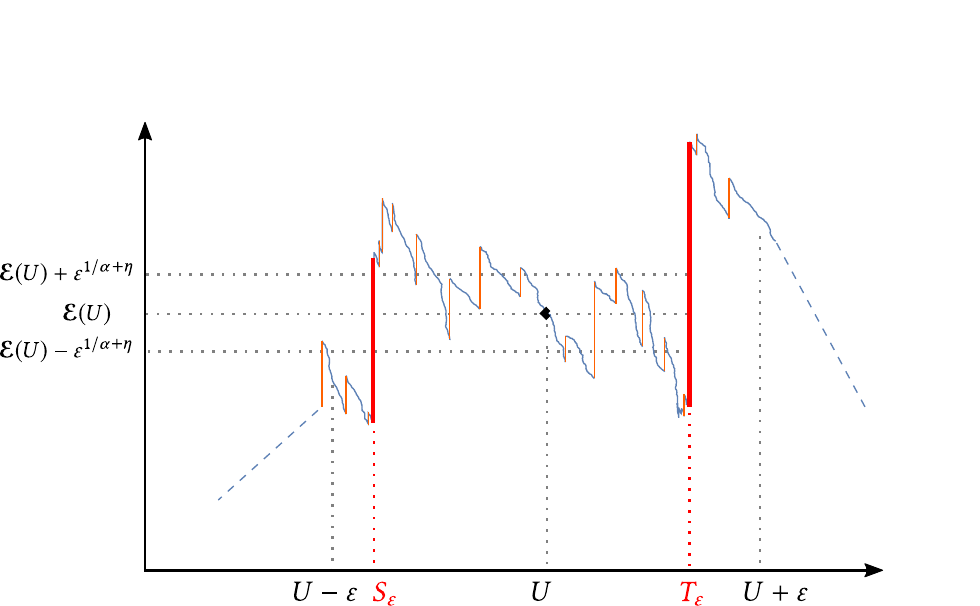}}
 \caption{Illustration of Lemma \ref{lem:blocking-jumps}.
On either side of $U$ there is a jump from a point at least
$\eps^{1/\alpha+\eta}$  below $\mathcal{E}(U)$  to a point at least 
$\eps^{1/\alpha+\eta}$ above.}
\label{f:jumps}
 \end{figure} 
In the conditions of the lemma, the two jumps at times $T_{ \varepsilon}$ and $S_{ \varepsilon}$ create vertical barriers so that  $ \mathbf{V}(s,U)\geq \eps^{1/\alpha+\eta}$
whenever $s\not\in[U-\eps,U+\eps]$. This implies the result.

To establish Lemma \ref{lem:blocking-jumps}
it suffices, as in \cite{curien:2014}, to consider the case when the
excursion $\mathcal{E}$ is replaced by an unconditioned
L\'evy process $(X_t:t\geq0)$ and $U$ is replaced by $0$.  The
lemma follows as in \cite{curien:2014} from an application of Borel--Cantelli along $ \varepsilon= 2^{-k}$ once we have proved that $\P(B_\eps^c)\leq C\eps^\gamma$ for some
$C,\gamma>0$, where 
\[
B_\eps=\{
\exists s\in[0,\eps]:
X_{s-}\in[-2\eps^{1/\alpha+\eta},-\eps^{1/\alpha+\eta}],
\Delta X_s\geq 3\eps^{1/\alpha+\eta}
\}.
\]
To prove the bound on $\P(B_\eps^c)$ we consider the excursions of
$(X_t:t\geq0)$ away from $0$.  (This differs from \cite{curien:2014}
where they consider excursions of $\overline X-X$ where $\overline X$
is the running supremum process.)  

Let $L_t$ be a local time of $X$ at $0$ and let $(g_j,d_j)$,
$j\in\mathcal I$, denote the excursion intervals.  Then, since $X$ has
only positive jumps, each $(g_j,d_j)$ contains a unique jump time
$h_j$ such that $X_{h_j-}<0$ and $X_{h_j}>0$.  The random measure 
\[
\sum_{j\in\mathcal I} \delta_{(L_{g_j},\Delta X_{h_j},-X_{h_j^-})}
\]
is a Poisson point process with intensity measure 
$dt\, \Pi(dx)\,  \mathbf{1}_{[0,x]}(r)dr$, see e.g.\ the remark after
Corollary 1 of \cite{bertoin:1992}.
Also, $L^{-1}$ is a subordinator \cite[Prop. V.4]{bertoin:1996} which
by the scaling property of $X$ is stable with index $1-1/\alpha$.
Therefore we can apply the same estimates as in \cite{curien:2014} to
obtain the required bound on $\P(B^c_\eps)$ and hence the result.
\end{proof}

\subsection{Point identification}
Our main result Theorem \ref{thm:V=D*} gives a pretty clear idea of
the metric in $\shred_\alpha$, but one could wonder whether the
quotient $ \mathbf{V}=0$ identifies more points than the trivial
identifications $\{ \mathbf{D}^\up =0\}$ and $ \{ \mathbf{D}^\down
=0\}$. The answer will require the use of the points of decrease of
L\'evy processes  studied by Bertoin in the 90's. \medskip

Recall that time $t \in \mathbb{R}$ is a (local) decrease 
time  ($f(t)$ is a decrease point) of a c\`adl\`ag function $f$ if for
some $ \varepsilon>0$ we have 
$$ f(t-x) \geq f(t) \geq f(t+x), \quad \mbox{ for all } x \in [0, \varepsilon].$$
Brownian motion almost surely has no decrease times (nor increase times)
by a famous result of Dvoretzky, Erd\H{o}s and Kakutani \cite{dvoretzky:1961}. However,
Bertoin \cite{bertoin:1994} has proved that spectrally positive stable
L\'evy processes almost surely possess decrease times but no increase
times. This was further studied in \cite{bertoin:1993,
  marsalle:1998}. The points of decrease a priori enable curves
$\gamma$ to "cross" the excursion $\ce$ and possibly perform
identifications for $ \mathbf{V}$ which were not permitted by $
\mathbf{D}^\up$ or $ \mathbf{D}^\down$ only. We will show that this is
not the case. The following result can be seen as an analog to the
point identification in the Brownian map \cite{legall:2007}.

\begin{theorem} \label{l:vd0}
Almost surely,  for all $s,t\in [0,1]$, if $\mathbf{V}(s,t)=0$ then either $\mathbf{D}^\up(s,t)=0$ or $ \mathbf{D}^\down(s,t)=0$.
\end{theorem}

\begin{proof}

If $0 \leq s \leq t \leq 1$ are such that $ \mathbf{V}(s,t)=0$, by
Lemma \ref{lem:min} this means that $h= \ce(s)= \ce(t)$ and that one
of the straight segments going from $(s, h)$ to $(t,h) $ or from
$(t,h) $ to $(s,h)$ around the cylinder, does not cross $
\mathbf{Slits}(\ce)$. To fix ideas, let us assume we are in the first
case. By local absolute continuity (see the proof
of Theorem \ref{thm:V=D*}), it is enough to argue with the L\'evy
process $X$.  

If $s<t$ are such that $ \mathbf{V}(s,t)=0$ in $X$ then in particular the segment $[s,t]\times \{h\} \subset \mathbb{R}^2$ does not cross $ \mathbf{Slits}(X)$. To begin with, let us show that   \begin{eqnarray} \label{eq:changesign} \mbox{$ X-h$ changes strict sign at most once on $[s,t]$}. \end{eqnarray}

 If $X-h$ changes sign a finite number $n$ of times on $[s,t]$, then if $n \geq 2$ there must be an increase  time of $X$ in $[s,t]$ which is excluded by \cite{bertoin:1994}. We just have to exclude the possibility that $n = \infty$.
 
\begin{figure}[!h]
 \begin{center}
 \includegraphics[width=1\linewidth]{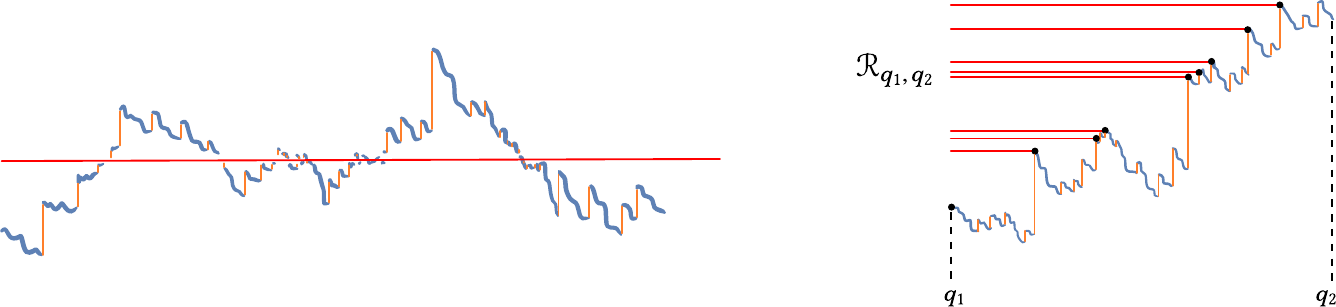}
 \caption{Left: Is it possible that the process $X$ oscillates around
   a horizontal line without jumping across it at any time? Right: the
   only possible heights not crossing $ \mathbf{Slits}(X)$ are among
   the red ones.} \label{fig:badline}

 \end{center}
 \end{figure}

To do so, fix two rational times $q_1<q_2$ and let us look at the possible levels $h\geq X_{q_1}$ so that the horizontal line at level $h$ does not intersect $ \mathbf{Slits}(X)$. Using Figure~\ref{fig:badline} (Right) it is easy to see that those possible heights are included in 
$$ \mathcal{R}_{q_1,q_2} =  \{ \overline{X}_s : s \in [q_1,q_2]\} \quad \mbox{ where }  \quad \overline{X}_s = \sup_{u \in [q_1,s]} X_u.$$
By standard results (see \cite[Lemma VIII.1]{bertoin:1996}) the set $
\mathcal{R}_{q_1,q_2}$ is a regenerative set, i.e.~a part of the range of a
subordinator. In our case, the subordinator is stable of index equal to $\alpha\rho=\alpha-1$ where $\rho =
\mathbb{P}(X_1>0)$ is the positivity parameter computed by Zolotarev's
formula, see \cite[p.~218]{bertoin:1996}. In particular this set has
Hausdorff dimension $\alpha-1 <1$.  In particular, for any $q_1<q_2<
q_3 < \cdots q_{2K}$ the intersection  
$$ \bigcap_{i=1}^{K-1} \mathcal{R}_{q_{2i-1},q_{2i}},$$
is an intersection of (part of) regenerative sets, which conditionally on their starting points $X_{q_1},
X_{q_3}, ... , X_{q_{2K-1}}$ are independent. Since those starting
points are almost surely distinct, it follows from
\cite[Example~1]{hawkes:1977} or \cite{bertoin:1999} that as soon as $K\geq 1+ \lceil
\frac{1}{2-\alpha}\rceil$ the intersection in the previous display is
almost surely empty.  

Performing the intersection over all countable choices of rationals
$q_1<q_2< q_3 < \cdots q_{2K}$ we deduce from the above consideration
that for any $h \in \mathbb{R}$, one cannot find rationals $q_1<q_2<
q_3 < \cdots <q_{2K}$ so that $X_{q_{2i-1}} < h < X_{q_{2i}}$ for $1
\leq i \leq K$ and so that $[q_1,q_{2K}] \times \{h \}$ does not
intersect $ \mathbf{Slits}(X)$. It follows that for any $s,t,h \in
\mathbb{R}$, if the segment $[s,t]\times \{h\} \subset \mathbb{R}^2$
does not cross $ \mathbf{Slits}(X)$, then $X-h$ changes strict sign at most
$K$ times. Together with the discussion just after \eqref{eq:changesign}, this finally proves \eqref{eq:changesign}. 

Coming back to our pair of identified points $s<t$, we deduce thanks
to \eqref{eq:changesign} that either $X-h$ does not change (strict)
sign, in which case we have $ \mathbf{D}^\up(s,t)=0$ or $
\mathbf{D}^\down(s,t)=0$, or otherwise it changes sign once and since
there are no increase times in $X$, the points must be identified
through a decrease time as is depicted in Figure~\ref{fig:V=0}. 

\begin{figure}[!h]
 \begin{center}
 \includegraphics[width=1\linewidth]{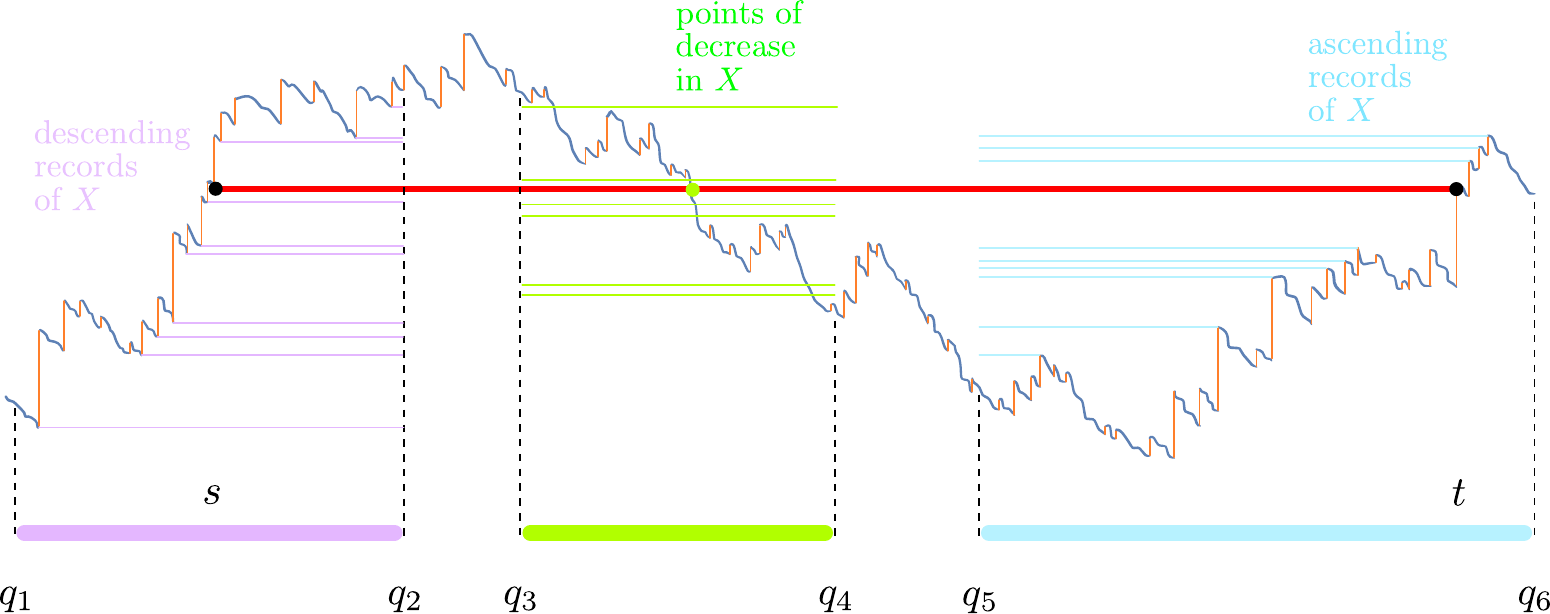}
 \caption{The last situation to treat: $s$ and $t$ are identified through a decrease point. \label{fig:V=0}}
 \end{center}
 \end{figure}
 In particular, we can find rational $q_1<s<q_2<q_3<q_4<q_5<q_6$ such that:
 \begin{itemize}
 \item level $h$ belongs to $
   \tilde{\mathcal{R}}_{q_1,q_2}=\{\underline{X}_u  : u \in
   [q_1,q_2]\}$ where $ \underline{X}_u = \inf_{v \in [u,q_2]} X_v$
   (purple in Figure~\ref{fig:V=0}),
  \item level $h$ belongs to $ \mathcal{R}_{q_5,q_6} =
    \{\overline{X}_u  : u \in [q_5,q_6]\}$ where $ \overline{X}_u =
    \sup_{v \in [q_5,u]} X_v$ (light blue in Figure~\ref{fig:V=0}),
    \item level $h$ belongs to  $ \mathcal{P}_{q_3,q_4} = \{ X_w :
      q_3<w<q_4,  \mbox{ such that }X_{a} \geq X_{w} \geq X_{b},
      \forall q_3\leq a\leq w \leq b \leq q_4\}$ (green in
      Figure~\ref{fig:V=0}). 
    \end{itemize}
    As above, the random sets $\tilde{\mathcal{R}}_{q_1,q_2} -X_{q_{2}}$, $ \mathcal{R}_{q_{5},q_{6}}-X_{q_{5}}$ and $  \mathcal{P}_{q_{3},q_{4}}-X_{q_{3}}$ are independent and of Hausdorff dimension respectively 
    $$ \mathrm{dim}(\tilde{\mathcal{R}}_{q_1,q_2}) =
    \mathrm{dim}(\mathcal{R}_{q_5,q_6}) = \alpha-1 \quad \mbox{ and }
    \quad \mathrm{dim}(\mathcal{P}_{q_3,q_4}) = 2- \alpha,$$ where the
    last dimension is computed in \cite{marsalle:1998}. The random sets $\tilde{\mathcal{R}}_{q_{1},q_{2}}-X_{q_{2}}$ and $ \mathcal{R}_{q_{5},q_{6}}-X_{q_{5}}$ are regenerative sets,  whereas $ \mathcal{P}_{q_{3},q_{4}} - \sup {\mathcal{P}_{q_{3},q_{4}}}$ is absolutely continuous with respect to a stable regenerative set, see \cite{marsalle:1998}. Since the sum of their codimensions is larger than $1$ we conclude as above
    that their intersections is almost surely empty. Performing an
    intersection over all possibles choices of the rationals, we
    deduce that the above situation cannot occur and the theorem is
    proved.
  \end{proof}

\subsection{Faces of $ \shred_\alpha$}
Recall that $\shred_\alpha$ may be realized as a weak limit of rescaled causal maps, according to Theorem \ref{thm:conv}. The positive jumps of the random path which encodes the causal map correspond to faces in the map (which appear as vertical duals of hard multimers). Therefore, in the continuum picture, we define a \emph{face} in $\shred_{\alpha}$ by relating it to a jump of the excursion of the  L{\'e}vy process $\ce$ which encodes $\shred_\alpha$. For each jump time $t$ of $\ce$  we
associate a face in $ \shred_{\alpha}$ of perimeter $2 \Delta
\ce(t)$ as follows: For each $s \in [
\ce(t-), \ce(t)]$ let $\ell_{t}(s)$
(resp.~$r_{t}(s)$) be the first instant on the left of $t$ (resp.~on
the right of $t$) such that  
$$ \ce(\ell_{t}(s)) = \ce(r_{t}(s)) =s.$$
To be more precise, we need to see $ \ce$ as indexed cyclically by
time and it may be that $\ell_{t}(s) > t$, but we always have the
cyclic ordering $\ell_{t}(s) \to t \to r_{t}(s)$. 

\begin{definition} The face $ \mathcal{F}_t$ in $ \shred_{\alpha}$ corresponding to a jump at $t$ is defined by $$ \mathcal{F}_t = \pi( \{ \ell_{t}(s) : s \in [ \ce(t-), \ce(t)]\}) \cup \pi( \{ r_{t}(s) : s \in [ \ce(t-), \ce(t)]\})$$
where $\pi: [0,1]\to \shred_\alpha$ denotes the projection to the quotient.
\end{definition}

The space $ \shred_\alpha$ possesses a countable number of faces. Due to the presence of decrease points in $\ce$, two faces may have a point in common as is indicated in Figure~\ref{fig:twofaces}. 
\begin {proposition} \label{prop:facestouch}
Almost surely, there exists a pair of faces in $\shred_\alpha$ which have points in common and the set of common points is a perfect set (every point in the set is a limit point of the set).
\end{proposition}
\begin{proof}
As before, we argue using the unconditioned L{\'e}vy process $X$. Let $t$ be a jump time of $X$ and let $\zeta$ be an independent exponential time with parameter 1. Define \emph{global decrease times} of $X$ on the interval $[t,t+\zeta]$ as those times $s\in[t,t+\zeta]$  at which
\begin{align*}
X_{s'} \geq X_s \geq X_{s''} \quad \text{for all} \quad s'\in[t,s] \quad \text{and}\quad s''\in [s,t+\zeta]
\end{align*}
and let $I$ be the set of all global decrease times on $[t,t+\zeta]$. By \cite{bertoin:1993}, the set $I$ is a perfect set and  $\mathbb{P}(I\neq \varnothing) >0$. On the event $I\neq \varnothing$, let $h\in I$. Then, since $t+\zeta$ is a stopping time, $X$ almost surely jumps accross level $h$ at some time $r> t+\zeta$. The set
\begin{align*}
\pi(\{s \in I~:~ X_{t-} \vee X_{r-} < X_s < X_{t} \wedge X_{r}\})
\end{align*}
then belongs to both the face $\mathcal{F}_t$ and $\mathcal{F}_r$ and is a perfect set.

\end{proof}
For the same reason, and due to cyclicity, it is even possible that a face $\mathcal{F}_t$ is adjacent to itself in the sense that $\ell_t(s) = r_t(s)$. Such points are global cut-points of the shredded sphere $\shred_{\alpha}$. 
\begin{figure}[!h]
 \begin{center}
 \includegraphics[width=0.8\linewidth]{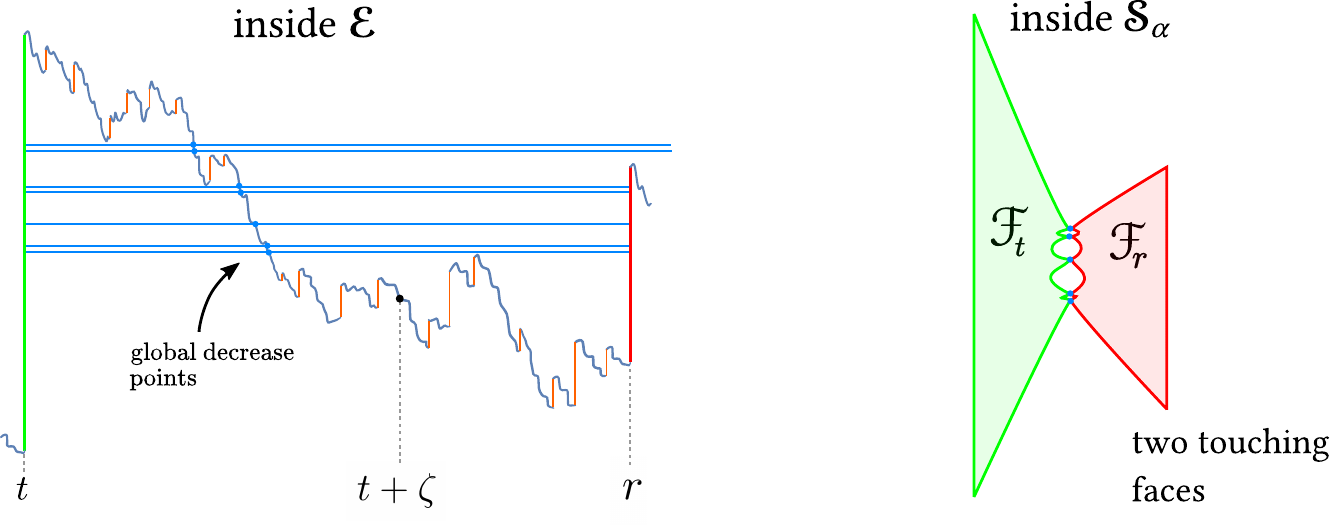}
 \caption{A time of decrease between two jumps may produce, inside $ \shred_\alpha$, two touching faces.}
 
 \label{fig:twofaces}
 \end{center}
 \end{figure}

Due to Proposition \ref{prop:facestouch}, it is natural to consider the graph $ \mathcal{G}_f$ formed by the faces
of $ \shred_\alpha$, where two faces are adjacent if they share a
common point. The connectedness of the $ \mathcal{G}_{f}$ is a question similar to
\cite[Question 11.2]{mating-of-trees}. We have not been able to find a complete solution, however, in the regime $\alpha$ close to  1, it is possible to adapt
the method of \cite{gwynne:2018} to show that the graph is connected
with probability one. The main idea is, for a given face
$\mathcal{F}$, to find another face $\mathcal{F}'$ in a `Markovian' way
and so that $\mathcal{F'}$ is typically larger than $\mathcal{F}$
(i.e.~so that the expectation of the logarithm of the ratio of their
lengths is positive). The interested reader may contact us for more
details. 

\subsection{The case $\alpha=2$}
\label{sec:comments}

The case $\alpha=2$ corresponds to  when $ \mu$ has finite
variance and this falls in the universality class of the generic
causal triangulations studied e.g. in \cite{curien:2017}. Although
\cite{curien:2017} deals with an infinite model, the results should
extend and $ n^{-1/2} \cdot \randdual$ should converge towards a segment of
height given by (a constant multiple of) the maximum of Brownian
excursion $ \mathbf{e}$. In the case of the Brownian excursion $
\mathbf{e}$, the definition of $ \mathbf{D}^*$ and  $\mathbf{V}$ also
make sense, but we trivially have  
$$ \mathbf{V}(s,t) = | \mathbf{e}(s)- \mathbf{e}(t)|,$$ since $
\mathbf{Slits}$ is empty. Moreover, Theorem \ref{thm:V=D*} still holds
since we can adapt the subadditive techniques of \cite[Section
2.3]{curien:2017} (using local time to measure horizontal distances)
and show that $ \mathbf{D}^*(s,t) = | \mathbf{e}(s)- \mathbf{e}(t)|$
as well. We refrain from doing so to keep the paper short.

\end{document}